\documentclass[11pt,a4paper]{amsart}

\usepackage{mathtools}
\usepackage{amsthm}
\usepackage{amssymb}
\usepackage{hyperref}
\usepackage{booktabs}

\usepackage{tikz}
\usepackage{tikz-cd}
\usetikzlibrary{positioning}

\tikzset{
  commutative diagrams/.cd,
  arrow style=tikz,
  diagrams={line width=.54pt},
  labels={font=\footnotesize},
}

\theoremstyle{definition}
\newtheorem{definition}{Definition}[section]
\newtheorem{example}[definition]{Example}
\newtheorem{question}[definition]{Question}
\newtheorem{remark}[definition]{Remark}

\theoremstyle{plain}
\newtheorem{proposition}[definition]{Proposition}
\newtheorem{lemma}[definition]{Lemma}
\newtheorem{theorem}[definition]{Theorem}
\newtheorem{corollary}[definition]{Corollary}

\numberwithin{equation}{section}

\def\fullref#1#2{%
  {#1\ \penalty 200\relax\ref{#2}}%
}

\newcommand{\gp}{\mathcal{P}}

\newcommand{\bbe}{\mathbb{E}}
\newcommand{\bbp}{\mathbb{P}}
\newcommand{\bbq}{\mathbb{Q}}
\newcommand{\bbr}{\mathbb{R}}
\newcommand{\bbs}{\mathbb{S}}
\newcommand{\bbf}{\mathbb{F}}

\newcommand\fdtpath[1]{#1}

\DeclarePairedDelimiter{\abs}{\lvert}{\rvert}

\DeclarePairedDelimiter{\parens}{\lparen}{\rparen}

\DeclarePairedDelimiter{\set}{\{}{\}}
\DeclarePairedDelimiterX{\gset}[2]{\{}{\}}{\,#1:#2\,}

\newcommand{\defterm}[1]{\textit{#1}}

\newcommand{\fsa}[1]{\mathcal{#1}}

\newcommand{\cgen}[1]{#1^{\#}}

\newcommand{\pres}[2]{\left\langle #1\:|\:#2 \right\rangle}
\newcommand{\lb}{\langle}
\newcommand{\rb}{\rangle}
\newcommand{\thue}{\leftrightarrow^*}

\newcommand{\nset}{\mathbb{N}}

\newcommand{\emptyword}{\varepsilon}
\newcommand{\rel}[1]{\mathcal{#1}}

\newcommand{\ZM}{{\mathbb{Z}M}}

\newcommand{\imreduces}{\rightarrow}

\newcommand{\reduces}{\rightarrow^*}
\newcommand{\reducesrev}{\leftarrow^*}

\newcommand{\equivrl}{\leftrightarrow}

\newcommand{\rev}{\mathrm{rev}}

\newcommand\fcrs{{\normalfont\scshape fcrs}}
\newcommand\fdt{{\normalfont\scshape fdt}}
\newcommand\nonfdt{{\normalfont\scshape nonfdt}}
\newcommand\biauto{{\normalfont\scshape biauto}}
\newcommand\auto{{\normalfont\scshape auto}}
\newcommand\nonauto{{\normalfont\scshape nonauto}}

\newcommand\aprop{{\normalfont\scshape a}}
\newcommand\bprop{{\normalfont\scshape b}}
\newcommand\cprop{{\normalfont\scshape c}}
\newcommand\dprop{{\normalfont\scshape d}}

\newcommand{\monly}{{\mathsf{M}}}
\newcommand{\m}[2]{\monly_{\text{#2}}^{\text{#1}}}

\newcommand{\powerset}{\mathbb{P}}

\let\cal\mathcal

\begin{document}

\title[On homogeneous monoids]{On finite complete rewriting systems, finite derivation type, and automaticity for
  homogeneous monoids}

\author{Alan J. Cain}
\address{%
Centro de Matem\'{a}tica e Aplica\c{c}\~{o}es\\
Faculdade de Ci\^{e}ncias e Tecnologia\\
Universidade Nova de Lisboa\\
2829--516 Caparica\\
Portugal
}
\email{%
a.cain@fct.unl.pt
}
\thanks{The first author was supported by an {\sc FCT} Ci\^{e}ncia 2008 fellowship and later by an Investigador
  {\sc FCT} fellowship ({\sc IF}/01622/2013/{\sc CP}1161/{\sc CT}0001).}

\author{Robert D. Gray}
\address{%
School of Mathematics\\
University of East Anglia\\
Norwich NR4 7TJ\\
United Kingdom
}
\email{%
Robert.D.Gray@uea.ac.uk
}
\thanks{The second author was partially supported by the {\scshape EPSRC} grant EP/N033353/1 `Special inverse
  monoids: subgroups, structure, geometry, rewriting systems and the word problem'.}

\author{Ant\'{o}nio Malheiro}
\address{%
Departamento de Matem\'{a}tica \&\ Centro de Matem\'{a}tica e Aplica\c{c}\~{o}es\\
Faculdade de Ci\^{e}ncias e Tecnologia\\
Universidade Nova de Lisboa\\
2829--516 Caparica\\
Portugal}
\email{%
ajm@fct.unl.pt
}
\thanks{This work was partially supported by the Funda\c{c}\~{a}o para a
Ci\^{e}ncia e a Tecnologia
(Portuguese   Foundation   for   Science   and   Technology)
through   the project
UID/MAT/00297/2013 (Centro de Matem\'{a}tica e Aplica\c{c}\~{o}es).}

\begin{abstract}
  This paper investigates the class of finitely presented monoids defined by homogeneous (length-preserving) relations
  from a computational perspective. The properties of admitting a finite complete rewriting system, having finite
  derivation type, being automatic, and being biautomatic are investigated for this class of monoids. The first main
  result shows that for any consistent combination of these properties and their negations, there is a homogeneous monoid
  with exactly this combination of properties. We then introduce the new concept of abstract Rees-commensurability (an
  analogue of the notion of abstract commensurability for groups) in order to extend this result to show that the same
  statement holds even if one restricts attention to the class of $n$-ary homogeneous monoids (where every side of every
  relation has fixed length $n$). We then introduce a new encoding technique that allows us to extend the result
  partially to the class of $n$-ary multihomogenous monoids.
\end{abstract}

\maketitle

\section{Introduction}
\label{sec_intro}

Numerous interesting algebras arise as semigroup algebras $K[S]$, where $K$ is a field and $S$ is a homogeneous
semigroup (that is, a semigroup that is defined by a presentation where all relations are length-preserving); examples
include algebras yielding set-theoretic solutions to the Yang--Baxter equation and quadratic algebras of skew type (see
for example \cite{Etingof1999, JespersBook2007, Cedo2010(3)} and \cite{cedo_alternatingtype, Gateva-Ivanova2003, Jespers2006}),
algebras related to Young diagrams, representation theory and algebraic combinatorics such as the plactic and Chinese
algebras (see \cite[Ch.~5]{lothaire_algebraic}, \cite{cedo_plactic, lascoux_plaxique} and \cite{cassaigne_chinese,
  jaszunska_chinese, cedo_minimal}), and algebras defined by permutation relations (see \cite{cedo_permutation, cedo_alternatingtype,
  cedo_abeliantype1}). In these examples, there are strong connections between the structure of the algebra $K[S]$ and that of
the underlying semigroup $S$.  Further motivation for studying this class comes from other important semigroups in the
literature that admit homogeneous presentations, such as the hypoplactic monoid \cite{novelli_hypoplactic}, shifted
plactic monoid \cite{serrano_shifted}, monoids with the same multihomogeneous growth as the plactic monoid
\cite{duchamp_placticgrowth}, trace monoids \cite{Diekert1997}, divisibility monoids \cite{Kuske2001}, queue monoids
\cite{huschenbett_queue}, and positive braid monoids \cite{Birman1998, dehornoy_gaussiangarside}.

When investigating a semigroup $S$ defined by homogeneous relations, and its associated semigroup algebra $K[S]$, a
useful first step is to find a good set of normal forms (canonical representatives over the generating set) for the
elements of the monoid, and thus for elements of the algebra. (See the list of open problems in
\cite[Section~3]{cedo_permutation} for more on the importance of this problem in the context of semigroups defined by
permutation relations.)  Specifically we would like a set of normal forms that is a regular language, and we want to be
able to compute effectively with these normal forms. Two situations where such a good set of normal forms does exist are
for monoids that admit presentations by finite complete rewriting systems (see \cite{book_srs}), and for monoids and
semigroups that are automatic (see \cite{epstein_wordproc,campbell_autsg}). Each of these properties also has
implications for properties of the corresponding semigroup algebra. Indeed, if the semigroup admits a finite complete
rewriting system, then the semigroup algebra admits a finite Gr\"{o}bner--Shirshov basis (see \cite{heyworth_rewriting}
for an explanation of the connection between Gr\"obner--Shirshov bases and complete rewriting systems), while the
automaticity of the semigroup implies that the algebra is an automaton algebra in the sense of Ufnarovskij; see
\cite{Ufnarovskij} and \cite[Section~1]{cedo_alternatingtype}.

Many of the examples of homogeneous semigroups mentioned above have been shown to admit presentations by finite complete
rewriting systems, and have been shown to be biautomatic; see for example \cite{cgm_plactic, cgm_chineseetc,
  cedo_grobner, guzelkarpuz_complete, kubat_grobner, chen_grobner}.  It is natural to ask to what extent these results
generalise to arbitrary homogeneous semigroups. One can ask: Does every homogeneous semigroup admit a presentation by a
finite complete rewriting system?  Is every such semigroup biautomatic? Within the class of homogeneous semigroups, what
is the relationship between admitting a finite complete rewriting system and being biautomatic? (For general semigroups,
these properties are independent; see \cite{otto_automonversus}). The aim of this paper is to make a comprehensive
investigation of these questions. In fact, we shall consider two different strengths of automaticity, called
automaticity and biautomaticity, and we shall also investigate the homotopical finiteness property of finite derivation
type (FDT) in the sense of Squier \cite{squier_finiteness}, which is a finiteness property that is satisfied by monoids that
admit presentations by finite complete rewriting systems (full definitions of all of these concepts will be given in
\fullref{Section}{sec_prelims}).

There are various degrees of homogeneity that one can impose on a semigroup presentation. We shall consider finite
presentations $\pres{A}{\rel{R}}$ which are:
\begin{itemize}
\item homogeneous: relations are length-preserving;
\item multihomogeneous: for each letter $a$ in the alphabet $A$, and for every relation $u=v$ in $\rel{R}$, the number
  of occurrences of the letter $a$ in $u$ equals the number of occurrences of the letter $a$ in $v$;
\item $n$-ary homogeneous: there is a fixed global constant $n$ such that for every relation $u=v$ in $\rel{R}$ the
  lengths of the words $u$ and $v$ are both $n$;
\item $n$-ary multihomogeneous: simultaneously $n$-ary homogeneous and multihomogeneous.
\end{itemize}
Of course, the most restricted class listed here is the class of $n$-ary multihomogeneous presentations.

For brevity, we introduce the following terminology for the four properties we are interested in: a monoid is
\begin{itemize}
\item \fcrs\ if it admits a presentation via a finite complete rewriting system (with respect to some finite generating set);
\item \fdt\ if it has finite derivation type;
\item \biauto\ if it is biautomatic;
\item \auto\ if it is automatic.
\end{itemize}
We will also use the natural negated terms: non-\fcrs, non-\fdt, non-\biauto, and non-\auto.

\begin{table}[t]
  \centering
  \begin{tabular}{c@{}c@{}cc@{}c@{}cll} %
    \toprule
    \fcrs & $\implies$ & \fdt & \biauto & $\implies$ & \auto & Example                                       & See                                           \\
    \midrule
    Y     &            & Y    & Y       &            & Y     & Plactic monoid                                & \cite{cgm_plactic}                            \\
    \midrule
    Y     &            & Y    & N       &            & Y     & $\m{\fcrs}{\auto}$                            & \fullref{Example}{eg:fcrsfdtnonbiautoauto}    \\
    Y     &            & Y    & N       &            & N     & $\m{\fcrs}{\nonauto}$                         & \fullref{Example}{eg:fcrsfdtnonbiautononauto} \\
    N     &            & Y    & Y       &            & Y     & $\m{\fdt}{\biauto}$                           & \fullref{Example}{eg:nonfcrsfdtbiautoauto}    \\
    N     &            & N    & Y       &            & Y     & $\m{\nonfdt}{\biauto}$                        & \fullref{Example}{eg:nonfcrsnonfdtbiautoauto} \\
    \midrule
    N     &            & Y    & N       &            & Y     & $\m{\fcrs}{\auto}\ast\m{\fdt}{\biauto} $      & \fullref{Section}{sec_freeproducts}           \\
    N     &            & Y    & N       &            & N     & $\m{\fdt}{\biauto}\ast\m{\fcrs}{\nonauto}$    & \fullref{Section}{sec_freeproducts}           \\
    N     &            & N    & N       &            & Y     & $\m{\fcrs}{\auto}\ast\m{\nonfdt}{\biauto}$    & \fullref{Section}{sec_freeproducts}           \\
    N     &            & N    & N       &            & N     & $\m{\fcrs}{\nonauto}\ast\m{\nonfdt}{\biauto}$ & \fullref{Section}{sec_freeproducts}           \\
    \bottomrule
  \end{tabular}
  \caption{Summary of examples of homogeneous monoids exhibiting all consistent combinations of the properties \fcrs, \fdt, \biauto, and \auto. Examples with the same combinations of properties also exist in the class of $n$-ary homogeneous monoids (see \fullref{Section}{sec:extending}).}
  \label{tb:resume_properties}
\end{table}

\begin{figure}[t]
  \centering
  \begin{tikzpicture}[every node/.style={align=center}]
    \path[rounded corners,fill=lightgray,opacity=.5] (2.75,5.9)--(-1.15,2)--(1,-0.15)--(4.9,3.75)--cycle;
    \path[rounded corners,fill=lightgray,opacity=.5] (2.5,5.5)--(-1,2)--(0,1)--(3.5,4.5)--cycle;
    \path[rounded corners,fill=lightgray,opacity=.5] (1.25,5.9)--(5.15,2)--(3,-0.15)--(-0.9,3.75)--cycle;
    \path[rounded corners,fill=lightgray,opacity=.5] (1.5,5.5)--(5,2)--(4,1)--(0.5,4.5)--cycle;
    \node[rotate=-45] at (2.85,4.85) {\fcrs};
    \node[rotate=-45] at (3.68,4.68) {\fdt};
    \node[rotate=45] at (1.15,4.85) {\biauto};
    \node[rotate=45] at (0.28,4.68) {\auto};
    \node[align=center,font=\small] (1) at (2,4) {Plactic\\[-1mm] monoid};
    \node(2) at (1,3) {$\m{\fcrs}{\auto}$} edge (1);
    \node(4) at (3,3) {$\m{\fdt}{\biauto}$} edge (1);
    \node(3) at (0,2) {$\m{\fcrs}{\nonauto}$} edge (2);
    \node(5) at (2,2) {$\ast$} edge (2) edge (4);
    \node(7) at (4,2) {$\m{\nonfdt}{\biauto}$} edge (4);
    \node(6) at (1,1) {$\ast$} edge (3) edge (5);
    \node(8) at (3,1) {$\ast$} edge (5) edge (7);
    \node(9) at (2,0) {$\ast$} edge (6) edge (8);
  \end{tikzpicture}
  \caption{The semilattice showing the relationship between examples.
    By taking the free product of two examples, one obtains a new monoid
    whose properties are given by taking the logical conjunction (that
    is, the `and' operation) of corresponding properties of the original
    example monoids. This corresponds to the meet operation in this semilattice.}
  \label{fig:examplerelationships}
\end{figure}
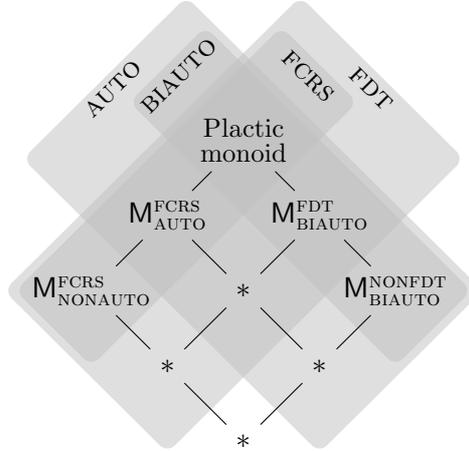

\begin{table}[t]
  \centering
  \begin{tabular}{c@{}c@{}cc@{}c@{}ccl} %
    \toprule
    \fcrs & $\implies$ & \fdt & \biauto & $\implies$ & \auto & Exists? & See                                             \\
    \midrule
    Y     &            & Y    & Y       &            & Y     & Y       & \cite{cgm_plactic}                              \\
    \midrule
    Y     &            & Y    & N       &            & Y     & Y       & \fullref{Theorem}{thm:from_homog_to_multihomog1} \\
    Y     &            & Y    & N       &            & N     & Y       & \fullref{Theorem}{thm:from_homog_to_multihomog1} \\
    N     &            & Y    & Y       &            & Y     & ?       & \fullref{Question}{question:multi_fdt_non-fcrs} \\
    N     &            & N    & Y       &            & Y     & Y       & \fullref{Theorem}{thm:from_homog_to_multihomog2} \\
    \midrule
    N     &            & Y    & N       &            & Y     & ?       & \fullref{Question}{question:multi_fdt_non-fcrs} \\
    N     &            & Y    & N       &            & N     & ?       & \fullref{Question}{question:multi_fdt_non-fcrs} \\
    N     &            & N    & N       &            & Y     & Y       & \fullref{Theorem}{thm:from_homog_to_multihomog2} \\
    N     &            & N    & N       &            & N     & Y       & \fullref{Theorem}{thm:from_homog_to_multihomog2} \\
    \bottomrule
  \end{tabular}
  \caption{Summary of the existence of examples of multihomogeneous monoids with consistent combinations of the properties
    \fcrs, \fdt, \biauto, and \auto. Examples with the same combinations of properties also exist in the class of $n$-ary
    multihomogeneous monoids.}
  \label{tbl:summarymultihomogeneous}
\end{table}

We are interested in which combinations of these properties a homogeneous monoid can have. Since in general \fcrs \
implies \fdt, and \biauto \ implies \auto, not all combinations will be possible. We refer to any combination of
properties that satisfies these restrictions, and that does not contain a property and its negation, as
\defterm{consistent}.  Our first main result shows that any consistent combination is possible within the class of
homogeneous monoids. We show this by constructing examples of homogeneous monoids with each consistent combination of
properties. We adopt the following naming scheme: in the example monoid $\m{\aprop}{\bprop}$, the superscript \aprop\
will be one of \fcrs, \fdt, or \nonfdt, indicating that the monoid is respectively \fcrs\ (and thus also \fdt), \fdt\
but not \fcrs, or non-\fdt\ (and thus also non-\fcrs); while the subscript \bprop\ will be one of \biauto, \auto, or
\nonauto, indicating that the monoid is respectively \biauto\ (and thus also \auto), \auto\ but not \biauto, or
non-\auto\ (and thus also non-\biauto). In \fullref{Section}{sec:fundexs}, we presents the fundamental examples
$\m{\fcrs}{\auto}$, $\m{\fdt}{\biauto}$, $\m{\fcrs}{\nonauto}$, and
$\m{\nonfdt}{\biauto}$. \fullref{Section}{sec_freeproducts} contains general results about the behaviour of the various
properties under free products of monoids, which we then use to construct the remaining examples. These results are
summarised in \fullref{Table}{tb:resume_properties} and the relationship between the various examples is illustrated in
\fullref{Figure}{tb:resume_properties}.

In \fullref{Sections}{sec:abstractrees} and \ref{sec:embedding} we introduce new concepts and prove new results, in
order to study the combinations of properties can occur in even more restricted classes. We first introduce and
investigate the notion of abstract Rees-commensurability (an analogue of abstract commensurability for groups
\cite[\S\S~iv.27ff.]{delaharpe_geometric}), which allows us to show that every consistent combination can arise within the
class of $n$-ary homogeneous monoids. (Thus \fullref{Table}{tb:resume_properties} and
\fullref{Figure}{tb:resume_properties} could also describe the situation for $n$-ary homogeneous monoids.) We then
develop a new encoding technique that embeds a homogeneous monoid into a $2$-generated multihomogeneous monoid. This
encoding technique allows us to obtain most of the consistent combinations of properties in the class of
multihomogeneous or $n$-ary multihomogenous monoids. Specifically, it allows us to construct $n$-ary multihomogeneous
monoids with any possile combination of the properties \fcrs, \biauto, and \auto, or any combination of the properties
\fdt, \biauto, and \auto. However, it does not allow us to construct examples to separate the properties \fcrs\ and
\fdt\ within the class of multihomogeneous or $n$-ary multihomogeneous monoids.
\fullref{Table}{tbl:summarymultihomogeneous} summarises the known
consistent combinations of properties in the class of multihomogeneous monoids. Using the results for abstract Rees
commensurability, the same table also describes the situation for the class of $n$-ary multihomogeneous monoids.

\section{Preliminaries}
\label{sec_prelims}

The subsection below on derivation graphs, homotopy bases and finite derivation type is self-contained, but it can be
complemented with \cite{otto_properties,Kobayashi5}. There is an alternative formulation of the same concepts in terms
of strict monoidal categories/groupoids and higher-dimensional variations of them, and homotopical algebra in higher
categories \cite{Lafont1995,GM2009}. However, our approach, using Squier complexes, is the same one used in papers by
Otto \cite{otto_modular,otto_modular_tr}, Wang \cite{wang_fcrs}, Pride and the second and third authors
\cite{gray_propertiesnotinherited}, and the third author \cite{malheiro_fdtlargeideals}; we will require methods and
results from these papers in \fullref{Sections}{sec_freeproducts}, \ref{sec:abstractrees}, and \ref{sec:embedding}.

For further information on automatic semigroups, see~\cite{campbell_autsg}. We assume familiarity with basic notions of
automata and regular languages (see, for example, \cite{hopcroft_automata}) and transducers and rational relations (see,
for example, \cite{berstel_transductions}), although we will recall some key results that we use frequently.  For
background on string rewriting systems we refer the reader to \cite{baader_termrewriting,book_srs}.

\subsection{Words, rewriting systems, and presentations}
\label{subsection_words}

We denote the empty word (over any alphabet) by $\emptyword$. For an
alphabet $A$, we denote by $A^*$ the set of all words over $A$. When
$A$ is a generating set for a monoid $M$, every element of $A^*$ can
be interpreted either as a word or as an element of $M$. For words
$u,v \in A^*$, we write $u=v$ to indicate that $u$ and $v$ are equal
as words and $u=_M v$ to denote that $u$ and $v$ represent the same
element of the monoid $M$. The length of $u \in A^*$ is denoted $|u|$,
and, for any $a \in A$, the number of symbols $a$ in $u$ is denoted
$|u|_a$. We denote by $u^\rev$ the reversal of a word $u$;
that is, if $u=a_1\cdots a_{n-1}a_n$ then $u^\rev =
a_na_{n-1}\cdots a_1$, with $a_i\in A$. If $\rel{R}$ is a relation on
$A^*$, then $\cgen{\rel{R}}$ denotes the smallest monoid congruence generated by
$\rel{R}$.

We use standard terminology and notation from the theory of string
rewriting systems; see \cite{book_srs} or \cite{baader_termrewriting}
for background reading.

If $M$ is a monoid, a presentation of $M$ is a pair $\pres{A}{\rel{R}}$ such that $M$ is isomorphic to the quotient
$A^*/\cgen{\rel{R}}$, in which case, the elements of $\rel{R}$ are called the defining relations. We write $[u]_M$ for
the set of words in $A^*$ equal to $u$ in $M$. The presentation $\pres{A}{\rel{R}}$ is \defterm{homogeneous}
(respectively, \defterm{multihomogeneous}) if for every $(u,v) \in \rel{R}$ and $a \in A$, we have $|u| = |v|$
(respectively, $|u|_a = |v|_a$). That is, in a homogeneous presentation, defining relations preserve length; in a
multihomogenous presentation, defining relations preserve the number of each symbol. A monoid is \defterm{homogeneous}
(respectively, \defterm{multihomogeneous}) if it admits a homogeneous (respectively, multihomogeneous)
presentation. Note that homogeneous and multihomogeneous presentations are not required to be finite presentations.

A \defterm{string rewriting system}, or simply a \defterm{rewriting
  system}, is a pair $(A,\rel{R})$, where $A$ is a finite alphabet and
$\rel{R}$ is a set of pairs $(\ell,r)$, usually written $\ell
\imreduces r$, known as \defterm{rewriting rules} or simply
\defterm{rules}, drawn from $A^* \times A^*$. The single reduction
relation $\imreduces_{\rel{R}}$ is defined as follows: $u
\imreduces_{\rel{R}} v$ (where $u,v \in A^*$) if there exists a
rewriting rule $(\ell,r) \in \rel{R}$ and words $x,y \in A^*$ such
that $u = x\ell y$ and $v = xry$. That is, $u \imreduces_{\rel{R}} v$
if one can obtain $v$ from $u$ by substituting the word $r$ for a
subword $\ell$ of $u$, where $\ell \imreduces r$ is a rewriting
rule. The reduction relation $\reduces_{\rel{R}}$ is the reflexive and
transitive closure of $\imreduces_{\rel{R}}$. The subscript $\rel{R}$
is omitted when it is clear from context. The process of replacing a
subword $\ell$ by a word $r$, where $\ell \imreduces r$ is a rule, is
called \defterm{reduction} by application of the rule $\ell \imreduces
r$; the iteration of this process is also called reduction. A word $w
\in A^*$ is \defterm{reducible} if it contains a subword $\ell$ that
forms the left-hand side of a rewriting rule in $\rel{R}$; it is
otherwise called \defterm{irreducible}.

The rewriting system $(A,\rel{R})$ is \defterm{finite} if both $A$ and
$\rel{R}$ are finite. The rewriting system $(A,\rel{R})$ is
\defterm{noetherian} if there is no infinite sequence $u_1,u_2,\ldots$ of words from
$A^*$ such that $u_i \imreduces u_{i+1}$ for all $i \in
\nset$. That is, $(A,\rel{R})$ is noetherian if any process of
reduction must eventually terminate with an irreducible word. The
rewriting system $(A,\rel{R})$ is \defterm{confluent} if, for any
words $u, u',u'' \in A^*$ with $u \reduces u'$ and $u
\reduces u''$, the pair $u'$ and $u''$ \defterm{resolves}, that is, there
exists a word $v \in A^*$ such that $u'
\reduces v$ and $u'' \reduces v$.
It is well known that a noetherian system is confluent if and only if
all critical pairs resolve, where critical pairs are obtained by considering
overlaps of left-hand sides of the rewrite rules in $\rel{R}$; see
\cite{book_srs} for more details.
A rewriting system
that is both confluent and noetherian is \defterm{complete}.
If a monoid admits a presentation with respect to some generating set $A$
that forms a finite complete rewriting system $\rel{R}$, the monoid is \fcrs.
In
that case, the irreducible elements form a set of unique normal forms, over
$A$, for the elements of the monoid.

The \defterm{Thue congruence} $\thue_{\rel{R}}$ is the equivalence
relation generated by $\imreduces_{\rel{R}}$. The elements of the monoid
presented by $\pres{A}{\rel{R}}$ are the $\thue_{\rel{R}}$-equivalence
classes. The relations $\thue_{\rel{R}}$ and $\cgen{\rel{R}}$ coincide.

Let $M$ be a homogeneous monoid. Let $\pres{A}{\rel{R}}$ be a
homogeneous presentation of $M$. Without lost of generality, we can assume
that $\rel{R}$ has no trivial relations of the form $a=a'$, for letters
$a$, $a'$ in $A$. Since none of the generators
represented by $A$ can be non-trivially decomposed, the alphabet $A$
represents a unique minimal generating set for $M$, and any generating
set must contain this minimal generating set. Any two words over $A$
representing the same element of $M$ must be of the same length. So
there is a well-defined function $\lambda : M \to \nset$ where
$x\lambda$ is defined to be the length of any word over $A$
representing $x$.
(Here and elsewhere we shall write the function symbol on the right.)
It is easy to see that $\lambda$ is a homomorphism. Following
\cite[Definition 2.1 in \S 4]{Droste2009} the function $\lambda$ is called a
\defterm{grading} of $M$, and so homogeneous monoids are \defterm{graded
monoids}.

\subsection{Derivation graphs, homotopy bases, and finite derivation type}

Associated with any monoid presentation $\pres{A}{\rel{R}}$ is a $2$-complex $\mathcal{D}$, called the \defterm{Squier
  complex}, whose $1$-skeleton has vertex set $A^*$ and edges corresponding to applications of relations from
$\mathcal{R}$, and that has $2$-cells adjoined for each instance of ``non-overlapping'' applications of relations from
$\rel{R}$ (see below for a formal definition of non-overlapping relations). The free monoid $A^*$ acts on $\mathcal{D}$
in a natural way via left and right multiplication. A collection of closed paths in $\mathcal{D}$ is called a
\defterm{homotopy base} if the complex obtained by adjoining cells for each of these paths, and those that they generate
under the action of the free monoid on the Squier complex, has trivial fundamental groups. A monoid defined by a
presentation is said to have \defterm{finite derivation type} (or \fdt\ for short) if the corresponding Squier complex
admits a finite homotopy base. It was shown by Squier \cite{squier_finiteness} that the property \fdt\ is independent of
the choice of a finite presentation, so we may speak of \fdt\ monoids. The original motivation for studying this notion
is Squier's result \cite{squier_finiteness} which says that if a monoid admits a presentation by a finite complete
rewriting system then the monoid must have finite derivation type. The study of these concepts is motivated further by
the fact that the fundamental groups of connected components of Squier complexes, which are called \defterm{diagram
  groups}, have turned out to be a very interesting class of groups; see \cite{Guba}. In recent important work
\cite{guiraud_higher}, acyclic polygraphs have been used to define a higher-dimensional homotopical finiteness condition
for higher categories. In particular, this work gives rise to a definition of $\text{\fdt}_n$ that extends the notion of
finite derivation type to arbitrary dimensions.

In more detail, with any monoid presentation $\gp = \pres{A}{\rel{R}}$ we associate a
graph (in the sense of Serre \cite{serre_trees}) as follows. The
\defterm{derivation graph} of $\gp $ is an infinite graph $\Gamma =
\Gamma(\gp) = (V,E,\iota, \tau, ^{-1})$ with \emph{vertex set} $V =
A^*$, and \emph{edge set $E$} consisting of the collection of
$4$-tuples
\[
\{ (w_1, r, \epsilon, w_2): \ w_1, w_2 \in A^*, r \in \rel{R}, \ \mbox{and} \
\epsilon \in \{ +1, -1 \} \}.
\]
The functions $\iota, \tau : E \rightarrow V$ associate with each edge
$\bbe = (w_1, r, \epsilon, w_2)$ (with $r=(r_{+1},r_{-1}) \in \rel{R}$) its
initial and terminal vertices $\iota \bbe = w_1 r_{\epsilon} w_2$ and
$\tau \bbe = w_1 r_{- \epsilon} w_2$, respectively. The mapping $^{-1}
: E \rightarrow E$ associates with each edge $\bbe = (w_1, r,
\epsilon, w_2)$ an inverse edge $\bbe^{-1} = (w_1, r, -\epsilon,
w_2)$.

A non-empty path in $\Gamma$ is a sequence of edges $\bbp = \bbe_1 \circ \bbe_2
\circ \ldots
\circ \bbe_n$, written in the diagrammatic order, where $\tau \bbe_i = \iota
\bbe_{i+1}$ for $i=1, \ldots,
      {n-1}$. Here $\bbp$ is a path from $\iota \bbe_1$ to $\tau
      \bbe_n$ and we extend the mappings $\iota$ and $\tau$ to paths
      by defining $\iota \bbp = \iota \bbe_1$ and $\tau \bbp = \tau
      \bbe_n$.  The inverse of a path $\bbp = \bbe_1 \circ \bbe_2
      \circ \ldots \circ \bbe_n$ is the path $\bbp^{-1} = \bbe_n^{-1}
      \circ \bbe_{n-1}^{-1} \circ \ldots \circ \bbe_1^{-1}$, which is
      a path from $\tau \bbp$ to $\iota \bbp$. A \emph{closed path} is
      a path $\bbp$ satisfying $\iota \bbp = \tau \bbp$. For two paths
      $\bbp$ and $\bbq$ with $\tau \bbp = \iota \bbq$ the composition
      $\bbp \circ \bbq$ is defined.

We denote the set of paths in $\Gamma$ by $P(\Gamma)$, where for each
vertex $w \in V$ we include a path $1_w$ with no edges, called the
\emph{empty path} at $w$.
The free monoid
$A^*$ acts on both sides of the set of edges $E$ of $\Gamma$ by
\[
x \cdot \bbe \cdot y = (x w_1, r, \epsilon, w_2 y)
\]
where $\bbe = (w_1, r, \epsilon, w_2)$ and $x, y \in A^*$. This
extends naturally to a two-sided action of $A^*$ on $P(\Gamma)$ where
for a path $\bbp = \bbe_1 \circ \bbe_2 \circ \ldots \circ \bbe_n$ we
define
\[
x \cdot \bbp \cdot y = (x \cdot \bbe_1 \cdot y) \circ (x \cdot \bbe_2 \cdot y)
\circ \ldots \circ (x \cdot \bbe_n \cdot y).
\]
If $\bbp$ and $\bbq$ are paths such that $\iota \bbp = \iota \bbq$ and
$\tau \bbp = \tau \bbq$ then we say that $\bbp$ and $\bbq$ are
\emph{parallel}, and write $\bbp \parallel \bbq$. We use $\parallel$ to denote
the subset of $P(\Gamma) \times P(\Gamma)$  of all pairs
of parallel paths.

An equivalence relation $\sim$ on $P(\Gamma)$ is called a
\emph{homotopy relation} if it is contained in $\parallel$ and
satisfies the following four conditions.

\begin{figure}
\centering
\begin{tikzcd}
\null & \iota \bbe_1 \iota \bbe_2 \arrow[swap]{dl}{\bbe_1 \cdot \iota \bbe_2}  \arrow{dr}{\iota \bbe_1 \cdot \bbe_2} & \\
\tau \bbe_1 \iota \bbe_2 \arrow[swap]{dr}{\tau \bbe_1 \cdot \bbe_2} & & \iota \bbe_1 \tau \bbe_2 \arrow{dl}{\bbe_1 \cdot \tau \bbe_2} \\
& \tau \bbe_1 \tau \bbe_2 &
\end{tikzcd}
\caption{Disjoint derivations in $\Gamma$.}
\label{fig_H1}
\end{figure}
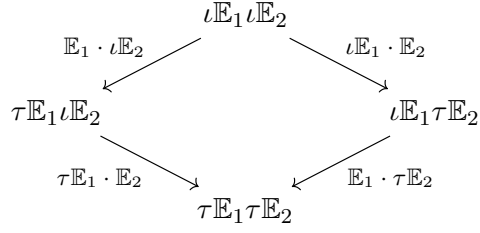

\begin{enumerate}
\item If $\bbe_1$ and $\bbe_2$ are edges of $\Gamma$, then \[
(\bbe_1 \cdot \iota \bbe_2) \circ (\tau \bbe_1 \cdot \bbe_2) \sim
(\iota \bbe_1 \cdot \bbe_2) \circ (\bbe_1 \cdot \tau \bbe_2 ).
\]
\item For any $\bbp, \bbq \in P(\Gamma)$ and $x,y \in A^*$
\[
\bbp \sim \bbq \ \ \mbox{implies} \ \  x \cdot \bbp \cdot y \sim x \cdot \bbq \cdot y.
\]
\item For any $\bbp, \bbq, \bbr, \bbs \in P(\Gamma)$ with $\tau \bbr = \iota \bbp = \iota \bbq$ and $\iota \bbs = \tau \bbp = \tau \bbq$
\[
\bbp \sim \bbq \ \ \mbox{implies} \ \ \bbr \circ \bbp \circ \bbs \sim \bbr \circ \bbq \circ \bbs.
\]
\item If $\bbp \in P(\Gamma)$ then $\bbp \bbp^{-1} \sim 1_{\iota \bbp}$, where $1_{\iota \bbp}$ denotes the empty path at the vertex $\iota \bbp$.
\end{enumerate}
The idea behind condition~1 is the following. Suppose that a word
$w$ has two disjoint occurrences of rewriting rules in the sense that
$ w = \alpha r_\epsilon \beta \alpha' r_{\epsilon'}' \beta' $ where
$\alpha, \beta, \alpha', \beta' \in A^*$, $r, r' \in R$ and $\epsilon,
\epsilon' \in \{ -1, +1 \}$. Let $\bbe_1 = (\alpha, r, \epsilon,
\beta)$ and $\bbe_2 = (\alpha', r', \epsilon', \beta')$. Then the
paths
\[
\bbp = (\bbe_1 \cdot \iota \bbe_2)\circ (\tau \bbe_1 \cdot \bbe_2),
\quad
\bbp' = (\iota \bbe_1 \cdot \bbe_2)\circ (\bbe_1 \cdot \tau \bbe_2)
\]
give two different ways of rewriting the word $w = \alpha r_\epsilon \beta \alpha' r_{\epsilon'}' \beta'$ to the word
$w = \alpha r_{-\epsilon} \beta \alpha' r_{-\epsilon'}' \beta'$, where in $\bbp$ we first apply the left-hand relation
and then the right-hand, while in $\bbp'$ the relations are applied in the opposite order; see
\fullref{Figure}{fig_H1}. We want to regard these two paths as being essentially the same, and this is achieved by
condition~1. Equivalent paths under this condition are said to be homotopic by \defterm{disjoint derivations}. This
relation is also often refereed to as the exchange relation or the interchange law in the literature; see \cite{GM2009}.

For a subset $C$ of $\parallel$, the homotopy relation \emph{$\sim_C$
  generated by $C$} is the smallest (with respect to inclusion)
homotopy relation containing $C$.
The homotopy relation generated by the empty set $\varnothing$ is
denoted by $\sim_0$.
If $\sim_C$ coincides with $\parallel$,
then $C$ is called a \defterm{homotopy base} for $\Gamma$. The
presentation $\pres{A}{\rel{R}}$ is said to have \emph{finite
  derivation type} if the derivation graph $\Gamma$
of $\pres{A}{\rel{R}}$ admits a finite homotopy base.  A finitely
presented monoid $M$ is said to have \emph{finite derivation type}, or
to be \fdt, if some (and hence any by \cite[Theorem~4.3]{squier_finiteness})
finite presentation for $M$ has finite derivation type.

It is not difficult to see that a subset $C$  of $\parallel$ is a homotopy base
of $\Gamma$ if and only if the set \[\{(\bbp \circ \bbq^{-1}, 1_{\iota \bbp}):
(\bbp,\bbq) \in C \}\] is a homotopy base for $\Gamma$. Thus we say that a set
$D$ of closed paths is a homotopy base if the corresponding set $\{ (\bbp,
1_{\iota \bbp}) : \bbp \in D \}$ is a homotopy base.

\subsection{Rational relations}

For references purposes, we briefly recall here some basic definitions and results regarding rational relations; we only
consider relations of the form $R \subseteq A^* \times B^*$. The set of rational relations between $A^*$ and $B^*$ is
the smallest subset of $\powerset{A^* \times B^*}$ that contains the empty set $\emptyset$, all singleton sets
$\set{(u,v)}$, and is closed under the operations of union, product, and the Kleene star
\[
X \mapsto X^* = \bigcup_{i=0}^\infty X^i.
\]
Note that the set of rational relations is also closed under the Kleene plus operation $X \mapsto X^+ = X^*X$.

\begin{proposition}[{\cite[Examples~5.1 \& 5.5]{berstel_transductions}}]
  \label{prop:idrational}
  For any regular language $L \subseteq A^*$, the relation
  \[
    \gset{(u,u)}{u \in L} \subseteq A^* \times A^*
  \]
  is rational.
\end{proposition}

\begin{proposition}
  \label{prop:intersectionwithregular}
  Let $K \subseteq A^*$ and $L \subseteq B^*$ be regular languages. If $R \subseteq A^* \times B^*$ is a rational
  relation, then $R \cap (K \times L)$ is a rational relation. In particular, $K \times L$ is a rational relation.
\end{proposition}

\begin{proof}
  Let $\fsa{T}$ be a transducer recognizing $R$ and let $\fsa{M}$ and $\fsa{N}$ be finite automata recognizing $K$ and
  $L$ respectively. Adapt $\fsa{T}$ to simulate $\fsa{M}$ and $\fsa{N}$ on the inputs from its first and second tapes,
  respectively, and to accept only if the simulated copies of $\fsa{M}$ and $\fsa{N}$ are in accept states. This adapted
  transducer recognizes $R \cap (K \times L)$.
\end{proof}

\subsection{Automaticity and biautomaticity}

\begin{definition}
Let $A$ be an alphabet and let $\$$ be a new symbol not in $A$. Define
the mapping $\textvisiblespace^\$ : A^* \times A^* \to ((A\cup\{\$\})\times (A\cup
\{\$\}))^*$ by
\begin{align*}
&(u_1\cdots u_m,v_1\cdots v_n) \mapsto\\
&\qquad\qquad\begin{cases}
(u_1,v_1)\cdots(u_m,v_n) & \text{if }m=n,\\
(u_1,v_1)\cdots(u_n,v_n)(u_{n+1},\$)\cdots(u_m,\$) & \text{if }m>n,\\
(u_1,v_1)\cdots(u_m,v_m)(\$,v_{m+1})\cdots(\$,v_n) & \text{if }m<n,
\end{cases}
\intertext{and the mapping ${}^\$\textvisiblespace : A^* \times A^* \to ((A\cup\{\$\})\times (A\cup \{\$\}))^*$ by}
&(u_1\cdots u_m,v_1\cdots v_n) \mapsto \\
&\qquad\qquad\begin{cases}
(u_1,v_1)\cdots(u_m,v_n) & \text{if }m=n,\\
(u_1,\$)\cdots(u_{m-n},\$)(u_{m-n+1},v_1)\cdots(u_m,v_n) & \text{if }m>n,\\
(\$,v_1)\cdots(\$,v_{n-m})(u_1,v_{n-m+1})\cdots(u_m,v_n) & \text{if }m<n,
\end{cases}
\end{align*}
where $u_i,v_i \in A$.
\end{definition}

\begin{definition}
\label{def:autstruct}
Let $M$ be a finitely generated monoid. Let $A$ be a finite set
of generators for $M$ and let $L \subseteq A^*$ be a regular language such
that every element of $M$ has at least one representative in $L$.  For
each $a \in A \cup \{\emptyword\}$, define the relations
\begin{align*}
L_a &= \{(u,v): u,v \in L, {ua} =_M {v}\}\\
{}_aL &= \{(u,v) : u,v \in L, {au} =_M {v}\}.
\end{align*}
The pair $(A,L)$ is an \defterm{automatic structure} for $M$ if
$L_a^\$$ is a regular language over $(A\cup\{\$\}) \times
(A\cup\{\$\})$ for all $a \in A \cup \{\emptyword\}$. A monoid $M$ is
\defterm{automatic}, or \auto, if it admits an automatic structure
with respect to some finite generating set.

The pair $(A,L)$ is a \defterm{biautomatic structure} for $M$ if
$L_a^\$$, ${}_aL^\$$, ${}^\$L_a$, and ${}_a^\$L$ are regular
languages over $(A\cup\{\$\}) \times (A\cup\{\$\})$ for all $a \in A
\cup \{\emptyword\}$. A monoid $M$ is \defterm{biautomatic}, or
\biauto, if it admits a biautomatic structure with respect to some finite
generating set. [Note that \biauto\ implies \auto.]
\end{definition}

Hoffmann \& Thomas have made a careful study of biautomaticity for
semigroups \cite{hoffmann_biautomatic}. They distinguish four notions
of biautomaticity for semigroups that require at least one of
$L_a^\$$ and ${}^\$L_a$ and at least one of ${}_aL^\$$ and
${}_a^\$L$ to be regular. These notions are all equivalent for
groups and more generally for cancellative semigroups
\cite[Theorem~1]{hoffmann_biautomatic} but distinct for semigroups
\cite[Remark~1 \& \S~4]{hoffmann_biautomatic}. \biauto\ clearly
implies all four Hoffmann--Thomas notions of biautomaticity. However,
we shall shortly prove that, within the class of homogeneous
monoids, any of the Hoffmann--Thomas notions of biautomaticity implies
\biauto\ (see \fullref{Proposition}{prop:htbiautoimpliesbiauto}).

In proving that $R^\$$ or ${}^\$R$ is regular, where $R$ is
a relation on $A^*$, a useful strategy is to prove that $R$ is a
rational relation (that is, a relation recognized by a finite
transducer \cite[Theorem~6.1]{berstel_transductions}) and then apply
the following result, which is a combination of
\cite[Corollary~2.5]{frougny_synchronized} and
\cite[Proposition~4]{hoffmann_biautomatic}:

\begin{proposition}
\label{prop:rationalbounded}
If $R \subseteq A^* \times A^*$ is rational relation and there is a
constant $k$ such that $\bigl||u|-|v|\bigr| \leq k$ for all $(u,v) \in
R$, then $R^\$$ and ${}^\$R$ are regular.
\end{proposition}

Now we shall prove some results on automaticity and biautomaticity for
the class of homogeneous monoids.

Unlike the situation for groups, both automaticity and biautomaticity
for monoids and semigroups are dependent on the choice of generating
set \cite[Example~4.5]{campbell_autsg}. However, for monoids,
biautomaticity and automaticity are independent of the choice of
\emph{semigroup} generating sets
\cite[Theorem~1.1]{duncan_change}. In our particular case of
homogeneous
monoids, we do have independence of the choice of generating set:

\begin{proposition}
\label{prop:autochangegen}
Let $M$ be a homogeneous monoid that is \auto\ (respectively,
\biauto). Then for any finite generating set $C$ of $M$ there is a language $K$
over $C$ such that $(C,K)$ is an
automatic (respectively, biautomatic) structure for $M$.
\end{proposition}

\begin{proof}
We first consider the case for \auto.
Suppose $(B,L)$ is an automatic structure for $M$.

Notice that both the alphabet $B$ and the alphabet $C$ must contain a
subalphabet representing the unique minimal generating set of
$M$. Without loss of generality, assume that they both contain the
alphabet $A$ representing this minimal generating set. For each $b \in
B$, let $w_b \in A^*$ be such that $w_b =_M b$. Let $\rel{Q} \subseteq
B^* \times A^*$ be the relation
\[
\gset[\big]{(b,w)}{b \in B}^*.
\]
Since $\rel{Q}$ is simply the subset of $B^* \times A^*$ obtained by taking the Kleene star of the finite set of
elements of the form $(b,w_b)$, it is by definition a rational relation.
Let
\[
K = L \circ \rel{Q} = \bigl\{v \in A^* : (\exists u \in L)\bigl((u,v) \in \rel{Q}\bigr)\bigr\}.
\]
Let $a \in A \cup \{\emptyword\}$. Then
\begin{align*}
(u,v) \in K_a &\iff u \in K \land v \in K \land ua =_{M} v\\
&\iff (\exists u',v' \in L)((u',u) \in \rel{Q} \land (v',v) \in \rel{Q} \land
u'a =_{M} v')\\
&\iff (\exists u',v' \in L)((u',u) \in \rel{Q} \land (v',v) \in \rel{Q}
\land (u',v') \in L_a)\\
&\iff (u,v) \in \rel{Q}^{-1} \circ L_a \circ \rel{Q}.
\end{align*}
Since a composition of rational relations is rational \cite[Theorem 4.4]{berstel_transductions}, it follows that
$K_a = \rel{Q}^{-1} \circ L_a \circ \rel{Q}$ is a rational relation. Furthermore
\begin{align*}
(u,v) \in K_a &\implies ua =_M v \\
&\implies (ua)\lambda = v\lambda \\
&\implies u\lambda + a\lambda = v\lambda \\
&\implies |u| + 1 = |v|;
\end{align*}
thus $K_a^\$$ is a regular language by
\fullref{Proposition}{prop:rationalbounded}.

For $c \in C - A$, let $u = u_1\cdots u_m \in A^*$ be such that $u =
c$; then $K_c^\$ = (K_{u_1}\circ K_{u_2}\circ\cdots \circ
K_{u_m})^\$$ is regular by \cite[Proposition~2.3]{campbell_autsg}
and similarly ${}^\$K_c$ is regular. Hence $(C,K)$ is an automatic
structure for $M$.

For \biauto, assume $(B,L)$ is a biautomatic structure for $M$ and
follow the above reasoning to show that each of the languages
$K_a^\$$, ${}_aK^\$$, ${}^\$K_a$, and ${}_a^\$K$ are regular.
\end{proof}

\begin{proposition}
\label{prop:htbiautoimpliesbiauto}
Let $M$ be a homogeneous monoid, let $B$ be a finite generating set for $M$, and let $L$ be a regular language over $B$
such that every element of $M$ has at least one representative in $L$ and such that, for each
$b \in B \cup \{\emptyword\}$, at least one of $L_b^\$$ and ${}^\$L_b$ and at least one of ${}_bL^\$$ and
${}_b^\$L$ is regular. Then $M$ is \biauto.
\end{proposition}

\begin{proof}
Suppose $L_b^\$$ and ${}_bL^\$$ are regular; the other cases are
similar.

As in the proof of \fullref{Proposition}{prop:autochangegen}, the alphabet $B$ must contain the unique minimal
generating set $A$ of $M$. Construct the relation $\rel{Q} \subseteq B^* \times A^*$ as in the proof of
\fullref{Proposition}{prop:autochangegen}. Let $K = L \circ \rel{Q}$.

Let $a \in A \subseteq B$. Then at least one of $L_a^\$$ and
${}^\$L_a$ and at least one of ${}_aL^\$$ and ${}_a^\$L$ is
regular. In particular, $L_a$ and ${}_aL$ are rational relations. So
$K_a = \rel{Q}^{-1} \circ L_a \circ \rel{Q}$ and ${}_aK =
\rel{Q}^{-1}\circ {}_aL \circ \rel{Q}$ are rational relations. If
$(u,v)$ is in $K_a$ or ${}_aK$, then $|u| +1 = |v|$. Hence $K_a^\$$,
${}^\$K_a$, ${}_aK^\$$, and ${}_a^\$K$ are all regular by
\fullref{Proposition}{prop:rationalbounded}. Since $a \in A \cup
\{\emptyword\}$ was arbitrary, this proves that $(A,K)$ is a
biautomatic structure for $M$.
\end{proof}

Despite the positive results obtained so far, note that \auto\ does
not imply \biauto\ in the class of homogeneous monoids, as
we shall see below in
\fullref{Example}{eg:fcrsfdtnonbiautoauto}.

\section{Fundamental examples}
\label{sec:fundexs}

\subsection{An \fcrs, \fdt, non-\biauto, \auto\ homogeneous monoid}

In this subsection we present a homogeneous monoid that is \fcrs\ and thus
\fdt, is \auto, but is not \biauto. By considering the reversal
semigroup of this example we will get a homogeneous monoid that admits
a finite complete rewriting system but is not automatic.

\begin{example}
\label{eg:fcrsfdtnonbiautoauto}
Let $\m{\fcrs}{\auto}$ be the monoid defined by the presentation $\pres{A}{\rel{R}}$,
where $A = \{a,b,c\}$ and $\rel{R}$ consists of the rewriting rules
\begin{align*}
cxyz &\imreduces cxcz,  & (x,y,z\in\{a,b\})  \\
cbca &\imreduces cacb.
\end{align*}
\end{example}

\begin{proposition}\label{proposition_eg:fcrsfdtnonbiautoauto_is_fcrs}
The monoid $\m{\fcrs}{\auto}$ is \fcrs.
\end{proposition}

\begin{proof}
The rewriting system $(A,\rel{R})$ is noetherian because every
rewriting rule either decreases the number of non-$c$ symbols, or it stays
the same and
decreases the number of symbols $b$ to the left of symbols $a$. To see
that it is confluent, notice that the only overlaps are those between
the left-hand side of $cbca \imreduces cacb$ and the left-hand side of
a rule of the form $cayz \imreduces cacz$, where $y,z \in
\{a,b\}$. However, they resolve since we have
\[
cbcayz \imreduces \begin{cases} cacbyz \imreduces cacbcz \\
cbcacz \imreduces cacbcz
\end{cases},\ \textrm{for any}\ y,z\in \{a,b\}.
\]
Therefore $(A,\rel{R})$ is confluent.
\end{proof}

\begin{proposition}
The monoid $\m{\fcrs}{\auto}$ is \auto, but non-\biauto.
\end{proposition}

\begin{proof}
Let $L$ be the language of normal forms of $(A,\rel{R})$. Since
$(A,\rel{R})$ is finite, $L$ is regular \cite[Lemma 2.1.3]{book_srs}. Let $u \in L$. Consider the
following cases separately:
\begin{enumerate}

\item $uc$ must also be in normal form, since no left-hand side of a
  rewriting rule ends in $c$. Hence
\[
L_c = \gset{(u,u)}{u \in L}(\emptyword,c),
\]
and so $L_c$ is rational by \fullref{Proposition}{prop:idrational}.

\item If $ub$ is not in normal form, then $ub$ must end with the
  left-hand side of a rewriting rule. Hence $u = u'cxy$ for some $x,y
  \in \{a,b\}$, then $ub = u'cxyb \imreduces u'cxcb$. This word
  $u'cxcb$ is in normal form since $u'cx$ (which is a prefix of $u$)
  is in normal form and no rewriting rule has left-hand side $cxcb$
  for any $x \in \{a,b\}$. Thus
\begin{align*}
  L_b ={}& \gset{(u,u)(\emptyword,b)}{u\in L, ub \in L} \\
         & \cup \gset{(u'cxy,u'cxcb)}{x,y \in \set{a,b}, u'cxy \in L} \\
  ={}& \gset{(u,u)(\emptyword,b)}{u\in L, ub \in L} \\
         & \cup \parens[\big]{\gset{(u',u')}{u'\in L}\gset{(cxy,cxcb)}{x,y \in \set{a,b}} \cap (L \times L)},
\end{align*}
and so $L_b$ is rational by \fullref{Proposition}{prop:intersectionwithregular}.

\item If $ua$ is not in normal form, then $ua$ must end with the
  left-hand side of a rewriting rule and so either $u = u'cbc$ or $u =
  u'cxy$ for some $x,y \in \{a,b\}$.
\begin{enumerate}
\item If $u = u''(cb)^\alpha y$, with $y\in \{a,b,c\}$ and where $\alpha \geq
1$ is maximal, then
  $ua = u''(cb)^\alpha ya \reduces u''ca(cb)^\alpha$, since
  $ua = u''(cb)^\alpha ya \imreduces u''(cb)^\alpha ca$ when $y\in \{a,b\}$,
and then $cbca\imreduces cacb$. Now, $u''ca(cb)^\alpha$ is in normal
  form since $u''$ and $ca(ab)^\alpha$ are in normal form and the only
  left-hand side of a rewriting rule of that ends in $ca$ is $cbca$,
  and $\alpha$ is maximal.
\item If $u = u'cay$, with $y\in\{a,b\}$, then $ua = u'caya \imreduces
u'caca$ and this
  word is in normal form since $u'ca$ is in normal form.
\end{enumerate}
Therefore
\begin{align*}
L_a ={}& \gset{(u,ua)}{u\in L, ua \in L} \\
&\cup \gset{(u''(cb)^\alpha y, u''ca(cb)^\alpha)}{\alpha \in \nset,y\in\set{a,b,c}, u''(cb)^\alpha y \in L, u'' \notin A^*cb} \\
&\cup \gset{(u'cay, u'caca)}{y \in \set{a,b}, u'cay \in L} \\
={}& \parens[\big]{\gset{(u,u))}{u\in L}(\emptyword,a) \cap (L\times L)} \\
&\cup \parens[\big]{\gset{(u''u'')}{u'' \in L\setminus(A^*cb)}(\emptyword,ca)(cb,cb)^+\gset{(y,\emptyword)}{y \in \set{a,b,c}} \cap L \times L} \\
&\cup \parens[\big]{\gset{(u',u')}{u' \in L}\gset{(cay, caca)}{y \in \set{a,b}} \cap L \times L},
\end{align*}
\end{enumerate}
and so $L_a$ is a union of relations, each of which is rational by \fullref{Proposition}{prop:idrational} and \ref{prop:intersectionwithregular}.

Note also that $L_\emptyword = \{(u,u) : u \in L\}$ is rational. Hence $L_x$ is a rational relation for any
$x \in A \cup \{\emptyword\}$. Moreover, if $(u,v)$ lies in one of these relations, then $\bigl||u| - |v|\bigr| \leq 1$
and so $L_x^\$$ is regular for all $x \in A \cup \{\emptyword\}$ by \fullref{Propositions}{prop:rationalbounded} and
\ref{prop:intersectionwithregular}. Hence $\m{\fcrs}{\auto}$ is \auto.

Suppose, with the aim of obtaining a contradiction, that $\m{\fcrs}{\auto}$ is \biauto. Then by
\fullref{Proposition}{prop:autochangegen} it admits a biautomatic structure $(A,L)$. Thus $L$ is a regular language
mapping onto $\m{\fcrs}{\auto}$ and ${}_cL^\$$ is regular. This contradicts
\fullref{Lemma}{lem:fcrsfdtnonbiautoautonotleftautomatic} below.
\end{proof}

\begin{lemma}
  \label{lem:fcrsfdtnonbiautoautonotleftautomatic}
  There is no regular language $L \subseteq A^*$ such that $L$ maps onto $\m{\fcrs}{\auto}$ and ${}_cL^\$$ is regular.
\end{lemma}

\begin{proof}
  Suppose, with the aim of obtaining a contradiction, that such a language $L$ exists. Then
  $({}_cL \circ {}_cL^{-1})^\$$ is regular. Let $n$ be an even number exceeding the number of states in an automaton
  recognizing $({}_cL \circ {}_cL^{-1})^\$$. Observe that
  \[
    {}_cL \circ {}_cL^{-1} = \{(u,v) \in L : cu =_\monly cv\}.
  \]
  Notice that $a^n b^{n+1}$ is not represented by any other word over $A$
  and similarly $b^na^nb$ is not represented by any other word over
  $A$. So $a^nb^{n+1},b^na^nb \in L$. Furthermore,
  \begin{align*}
    ca^nb^{n+1} &\reduces (ca)^{n/2}(cb)^{(n/2)+1} \\
    \intertext{and}
    cb^na^{n}b &\reduces (cb)^{n/2}(ca)^{n/2}cb \\
                &\reduces (ca)^{n/2}(cb)^{(n/2)+1}
  \end{align*}
  and so $(a^nb^{n+1},b^na^nb) \in {}_cL \circ {}_cL^{-1}$. Since $n$ exceeds the number of states in an automaton
  recognizing $({}_cL \circ {}_cL^{-1})^\$$, we can apply the pumping lemma to a segment of the word that lies within
  the first $n$ letters of $(a^nb^{n+1},b^na^nb)^\$$ (that is, to a subword of the form $(a^k,b^k)^\$$ for some
  $k \geq 1$) to see that $(a^{n+ik}b^{n+1},b^{n+ik}a^nb) \in {}_cL \circ {}_cL^{-1}$ for some $k \geq 1$ and for all
  $i \in \nset$. Hence, by definition of the relation ${}_cL$, we have $ca^{n+ik}b^{n+1} =_\monly cb^{n+ik}a^nb$. But
  \begin{align*}
    ca^{n+2k}b^{n+1} &\reduces (ca)^{n/2+k}(cb)^{(n/2)+1} \\
    \intertext{and}
    cb^{n+2k}a^{n}b &\reduces (cb)^{n/2+k}(ca)^{n/2}cb \\
                     &\reduces (ca)^{n/2}(cb)^{(n/2)+k+1},
  \end{align*}
  and so the normal forms of $ca^{n+ik}b^{n+1}$ and $cb^{n+ik}a^nb$ are unequal; this contradicts the previous
  equality. Thus $M$ is not \biauto.
\end{proof}

\subsection{An \fcrs, \fdt, non-\biauto, non-\auto\ homogeneous monoid}

\begin{definition}
Let $S$ be a monoid defined by a presentation
$\pres{A}{\rel{R}}$. Denote by $M^\rev$ the monoid defined by the
presentation $\pres{A}{\rel{R}^\rev}$, where
$\rel{R}^\rev=\{(l^\rev,r^\rev): (l,r)\in\rel{R}\}$, that is called
the \defterm{reversal monoid} of $M$. [Note that $(M^\rev)^\rev \simeq M$.]
\end{definition}

\begin{example}
\label{eg:fcrsfdtnonbiautononauto}
Let $\m{\fcrs}{\nonauto} = \parens[big]{\m{\fcrs}{\auto}}^\rev$, which is defined by the presentation
$\pres{A}{\rel{R}^\rev}$, where $\pres{A}{\rel{R}}$ is the
presentation defining \fullref{Example}{eg:fcrsfdtnonbiautoauto}.
\end{example}

Since $\m{\fcrs}{\nonauto}$ is presented by $\pres{A}{\rel{R}^\rev}$ we can argue as in
the proof of Proposition~\ref{proposition_eg:fcrsfdtnonbiautoauto_is_fcrs}
 that the rewriting system is noetherian and that the overlaps, that
result in critical pairs, resolve in a similar way. Thus $\m{\fcrs}{\nonauto}$ is also
\fcrs\ and thus \fdt.

\begin{proposition}
  $\m{\fcrs}{\nonauto}$ is non-\auto.
\end{proposition}

\begin{proof}
  Suppose, with the aim of obtaining a contradiction, that $\m{\fcrs}{\nonauto}$ is \auto. Let $(A,L)$ be an automatic structure
  for $\m{\fcrs}{\nonauto}$. Then $L_a^\$$ is regular for all $a \in A \cup \set{\emptyword}$. Since $\m{\fcrs}{\nonauto}$ is homogeneous, if
  $(u,v) \in L_a$, then $\abs[\big]{|u|-|v|} \leq 1$ and so ${}^\$L_a$ is regular by
  \cite[Corollary~4.2]{hoffmann_notions}. Notice that $({}^\$L_a)^\rev = \parens[\big]{{}_a(L^\rev)}^\$$. Hence
  $L^\rev$ is a regular language mapping onto $\parens[\big]{\m{\fcrs}{\nonauto}}^\rev$ such that $\parens[\big]{{}_a(L^\rev)}^\$$ is regular. Since
  $\parens[\big]{\m{\fcrs}{\nonauto}}^\rev \simeq \m{\fcrs}{\auto}$, this contradicts \fullref{Lemma}{lem:fcrsfdtnonbiautoautonotleftautomatic} and so
  $\m{\fcrs}{\nonauto}$ is indeed non-\auto.
\end{proof}

\subsection{A non-\fcrs, \fdt, \biauto, \auto\ homogeneous monoid}

The following homogeneous monoid was introduced by Katsura and
Kobayashi \cite[Example~3]{katsura_nofcrs}, who showed that it
is non-\fcrs, but is \fdt. We shall prove that it is \biauto\ and thus
\auto.

\begin{example}
\label{eg:nonfcrsfdtbiautoauto}
Let $A = \{a,b_i,c_i,d_i : i = 1,2,3\}$ and let $\rel{R}$ consist of
the rewriting rules
\begin{align}
b_ia &\imreduces ab_i && i \in \{1,2,3\}, \label{eq:katsura:ba}\\
c_jb_j &\imreduces c_1b_1 && j \in \{2,3\}, \label{eq:katsura:cb}\\
b_jd_j &\imreduces b_1d_1 && j \in \{2,3\}. \label{eq:katsura:bd}
\end{align}
Let $\m{\fdt}{\biauto} = \pres{A}{\rel{R}}$. Then $\m{\fdt}{\biauto}$ is
\fdt\ \cite[\S~4]{katsura_nofcrs} but is
non-\fcrs\ \cite[Proposition~3]{katsura_nofcrs}.
\end{example}

\begin{proposition}
The monoid $\m{\fdt}{\biauto}$ of \fullref{Example}{eg:nonfcrsfdtbiautoauto} is \biauto\ and thus \auto.
\end{proposition}

\begin{proof}
Let $\rel{S}$ consist of the following rewriting rules:
\begin{align}
b_ia &\imreduces ab_i && i \in \{1,2,3\}, \label{eq:katsura2:ba} \\
c_ja^kb_j &\imreduces c_1a^kb_1 && j \in \{2,3\}, k \in \nset \cup \{0\}, \label{eq:katsura2:cab} \\
c_1a^kb_1 d_j &\imreduces c_ja^{k}b_1d_1 && j \in \{2,3\}, k \in
\nset \cup \{0\}, \label{eq:katsura2:cabd} \\
b_jd_j &\imreduces b_1d_1 &&  j \in \{2,3\}.\label{eq:katsura2:bd}
\end{align}
Notice that every rule in $\rel{S}$ is a consequence of those in $\rel{R}$.
Indeed, using rules in $\rel{R}$, we have
\begin{multline*}
c_ja^kb_j \equivrl c_ja^{k-1}b_ja \equivrl \ldots \equivrl c_jb_ja^k
\equivrl c_1b_1a^k \equivrl c_1ab_1a^{k-1} \equivrl \ldots \equivrl c_1a^kb_1;
\end{multline*}
and
\begin{multline*}
c_1a^kb_1 d_j \equivrl \ldots \equivrl c_1b_1a^{k}d_j
\equivrl c_jb_ja^{k}d_j \equivrl \ldots \equivrl c_ja^{k}b_jd_j
\equivrl c_ja^{k}b_1d_1.
\end{multline*}

Consider any ordering $<$ of $A$ satisfying
$b_1 < b_j < a < d_1 < d_j$ (for $j \in \{2,3\}$). By
\cite[Lemma 2.4.3]{baader_termrewriting} the right-to-left
length-plus-lexicographic order induced by $<$ is noetherian. Moreover, any
rewriting using a rule in $\rel{S}$
decreases a word with respect to this ordering. Thus the rewriting system
$(A,\rel{S})$ is noetherian.

To see that $(A,\rel{S})$ is confluent, notice that there are two
possible overlaps of left-hand sides of rewriting rules: an overlap of
\eqref{eq:katsura2:ba} and \eqref{eq:katsura2:cab}, and an overlap of
 \eqref{eq:katsura2:cab} and \eqref{eq:katsura2:bd}.
However, critical pairs resolve, since
\begin{align*}
c_ja^kb_ja &\imreduces \begin{cases}
c_ja^{k+1}b_j \imreduces c_1a^{k+1}b_1 \\
c_1a^kb_1a \imreduces c_1a^{k+1}b_1
\end{cases}
\intertext{and}
c_ja^kb_jd_j &\imreduces \begin{cases}
c_ja^kb_1d_1 \\
c_1a^kb_1d_j \imreduces c_ja^kb_1d_1
\end{cases}.
\end{align*}

Let $L$ be the language of $\rel{S}$-irreducible words of $A^*$. That is,
\begin{align*}
L = A^* - A^*\Big(&\{b_1a,b_2a,b_3a\} \cup \{c_2,c_3\}a^*\{b_2,b_3\} \\
&\qquad\cup c_1a^*b_1\{d_2,d_3\} \cup \{b_2d_2,b_3d_3\}\Big)A^*;
\end{align*}
thus $L$ is regular. To prove that $(A,L)$ is an automatic structure for $\m{\fdt}{\biauto}$, we first show that $L_x$ and ${}_xL$ are
rational relations for all $x \in A\cup \set{\emptyword}$. Since $L$ is a cross-section for $\m{\fdt}{\biauto}$, the relations
$L_\emptyword$ and ${}_\emptyword L$ are both equal to the equality relation and hence are trivially rational.

So let $x \in A$ and $w \in L$. Suppose first that $x \in \set{b_2,b_3}$, then there may be a left-hand side of a
rewriting rule of type \eqref{eq:katsura2:cab} at the rightmost end of the word $wx$. In this case, $w$ must be of the
form $w'c_ja^k$ for some $w' \in L$ (since a prefix of a irreducible word is irreducible), $j \in \set{2,3}$ and
$k \in \nset\cup\set{0}$; applying the rewriting rule yields $w'c_1a^kb_1$, which is in normal form since no left-hand
side of a rewrite rule contains $c_1$ except for \eqref{eq:katsura2:cabd}, which clearly cannot be applied. So when
$x \in \set{b_2,b_3}$, at most one application of a rewrite rule at the rightmost end turns $wx$ into a normal form word. Hence
\begin{align*}
  L_{b_j}
  ={}& \gset[\big]{(w,wx)}{w \in L \setminus Lc_ja^*} \\
  &\cup \gset[\big]{(w'c_ja^k,w'c_1a^kb_1)}{w' \in L, k \in \nset\cup\set{0}} \\
  ={}& \gset[\big]{(w,w)}{w \in L \setminus Lc_ja^*}(\emptyword,x) \\
  &\cup \gset[\big]{(w',w')}{w' \in L}(c_j,c_1)(a,a)^*(\emptyword,b_1).
\end{align*}
Hence each $L_{b_j}$ is a rational relation by \fullref{Proposition}{prop:intersectionwithregular}.

Suppose now that $x \in \set{d_2,d_3}$. Reasoning similar to the previous paragraph shows that $wx$ is in normal form,
or one application of a rewrite rule of type \eqref{eq:katsura2:cabd} or \eqref{eq:katsura2:bd} turns it into normal
form. (Note that an application of a rule of type \eqref{eq:katsura2:bd} might be followed by one of type
\eqref{eq:katsura2:cab}, but these can be replaced by one of type \eqref{eq:katsura2:cabd}.) Hence
\begin{align*}
  L_{d_j}
  ={}& \gset[\big]{(w,wx)}{w \in L \setminus (Lc_1a^*b_1 \cup Ld_2 \cup Ld_3)} \\
  & \cup \gset[\big]{(w'c_1a^kb_1,w'c_ja^kb_1d_1)}{w' \in L, k \in \nset\cup\set{0}} \\
  & \cup \gset[\big]{(w'b_j,w'b_1d_1)}{w' \in L \setminus Lc_1a^*} \\
  ={}& \gset[\big]{(w,w)}{w \in L \setminus (Lc_1a^*b_1 \cup Ld_2 \cup Ld_3)}(\emptyword,x) \\
  & \cup \gset[\big]{(w',w')}{w' \in L}(c_1,c_j)(a,a)^*(b_1,b_1)(\emptyword,d_1) \\
  & \cup \gset[\big]{(w',w')}{w' \in L \setminus Lc_1a^*}(b_j,b_1)(\emptyword,d_1).
\end{align*}
Hence each $L_{d_j}$ is a rational relation by \fullref{Proposition}{prop:intersectionwithregular}.

Suppose that $x = a$. Then the only rewriting rules that can apply to $wx$ are a sequence of rule of type
\eqref{eq:katsura2:ba}, rewriting $w'yb_{i_1}\cdots b_{i_k}a$ (where $y \notin \set{b_1,b_2,b_3}$) to
$w'yab_{i_1}\cdots b_{i_k}$. This resulting word is in normal form, since the only way a rewriting rule could apply was
if $y=c_{i_1}$, but this means the word $w$ would contain $c_{i_1}b_{i_1}$, which contradicts $w\in L$. Hence
\begin{align*}
L_a &= \gset[\big]{(w'yb_{j_1}\cdots b_{j_k},w'yab_{j_1}\cdots b_{j_k})}{w'y \in L, b_{j_1},\ldots,b_{j_k} \in \set{1,2,3}, k\in\nset\cup\set{0}} \\
&= \gset[\big]{(w'y,w'y)}{w'y \in L}(\emptyword,a)\gset{(b_j,b_j)}{j \in \set{1,2,3}}^*.
\end{align*}

Finally, if $x \in \set{b_1,c_1,c_2,c_3,d_1}$, then $wx$ is already in normal form: hence, in this case,
\begin{align*}
 L_x = \gset{(w,w)(\emptyword,x)}{w \in L}.
\end{align*}

In each case, $L_x$ is a rational relation. Since $\m{\fdt}{\biauto}$ is homogeneous, if $(u,v) \in L_x$ for $x \in A$, then
$|v| = |ua| = |u| + 1$. Furthermore, if $(u,v) \in L_\emptyword$, then $|u| = |v|$. Hence $L_x^\$$ and ${}^\$L_x$ are
regular for all $x \in A \cup \set{\emptyword}$ by \fullref{Proposition}{prop:rationalbounded}.

Similar reasoning shows that ${}_x L$ is a rational relation: if $x \in \set{b_1,b_2,b_3}$ and $w \in L$, then rewriting
$xw$ to normal form can consist of a sequence of applications of rules of type \eqref{eq:katsura2:ba} followed possibly
by one of type \eqref{eq:katsura:bd}; for all other $x$, at most one rewriting rule is
required. \fullref{Proposition}{prop:rationalbounded} then applies to show that ${}_xL^\$$ and ${}_x^\$L$ are regular.

Hence $(A,L)$ is a biautomatic structure for $\m{\fdt}{\biauto}$.
\end{proof}

\subsection{A non-\fcrs, non-\fdt, \biauto, \auto\ homogeneous monoid}

In this section we give an example of a homogeneous monoid that
is non-\fdt\ and thus non-\fcrs, but which is \biauto\ and thus \auto.

\begin{example}
\label{eg:nonfcrsnonfdtbiautoauto}
Let $A=\{a,b\}$ and let $\rel{R}$ be the rewriting system on $A\cup\{c\}$ consisting of the three rules:
\[
\begin{array}{lclcll}
K_a &:& ac & \imreduces & ca &  \\
K_b &:& bc & \imreduces & cb &  \\
C &:& cab& \imreduces & cbb. &
\end{array}
\]
Let ${\gp}$ be the presentation $\pres{A\cup\{c\}}{\rel{R}}$, and let $\m{\nonfdt}{\biauto}$ be the monoid presented by $\gp$.
\end{example}

\begin{theorem}
\label{thm_exampleNotFDT}
The monoid $\m{\nonfdt}{\biauto}$:
\begin{enumerate}
\item has $A^* \cup c^+b^*a^*$ as set of unique normal forms (that is, it is a
 set,  over the generating
set $A$, in one-to-one correspondence with the elements of $\m{\nonfdt}{\biauto}$);
\item is \biauto\ and thus \auto;
\item is non-\fdt\ and thus non-\fcrs.
\end{enumerate}
\end{theorem}

Part~1 of \fullref{Theorem}{thm_exampleNotFDT} will follow from \fullref{Lemma}{lem_infCRS} below, and part~2 is proved in \fullref{Lemma}{lem_bio}. Then, the rest of the subsection will be devoted to proving that $\m{\nonfdt}{\biauto}$ is non-\fdt, thus establishing part~3.

\begin{remark}
The methods we use here to prove that \fullref{Example}{eg:nonfcrsnonfdtbiautoauto} is not-\fdt \ are similar to those used in the proof of \cite[Theorem~1]{gray_propertiesnotinherited}. In particular we will use the notion of critical peaks, and resolution of critical peaks, in our proof. We refer the reader to \cite[Section~2]{gray_propertiesnotinherited} for the definitions of these concepts, and their connection with complete rewriting systems
and \fdt.
\end{remark}

Let us begin by fixing some of
the notation.  We start by
adding to $\gp$ infinitely many rules of the form
\[
\overline{C}_u \; : \; cuab \imreduces cubb \; (u\in A^*)
\]
and denote by $\rel{R}'$ the set of all these rules.  Notice first that
$\overline{C}_{\epsilon}$ is precisely the rule $C$ defined above and
that, for any word $u\in A^*$ the words $cuab$ and $cubb$ represent
the same element of the monoid $\m{\nonfdt}{\biauto}$, since in the word $cuab$ we can
use relations of the form $K_x$ to pass the letter $c$ through the word $u$
from left to right,
then replace $cab$ by $cbb$ using the relation $C$, and finally move the $c$ back through $u$ again
from right to left
using
the relations $K_x$.
It follows that the presentations ${\gp} = \pres{A\cup\{c\}}{\rel{R}}$ and
$\overline{{\gp}}=\pres{A\cup\{c\}}{\rel{R} \cup \rel{R}'}$
are equivalent presentations, in the sense that two words $u, v \in (A\cup\{c\})^*$ are equivalent modulo the relations $\rel{R}$ if and only if they are equivalent modulo the relations $\rel{R} \cup \rel{R}'$. In particular, the monoid $\m{\nonfdt}{\biauto}$ is also defined by the infinite presentation
\[
\overline{{\gp}}=\pres{A\cup\{c\}}{\rel{R} \cup \rel{R}'}.
\]

\begin{lemma}
\label{lem_infCRS}
The infinite presentation $\overline{\gp}$ is a complete presentation of $\m{\nonfdt}{\biauto}$.
The set of irreducible words with respect to this complete rewriting system is
\[
  A^* \cup c^+b^*a^* = A^* \cup \gset{c^ib^jc^l}{i \in \nset, j,k \in \nset \cup \set{0}}.
\]
\end{lemma}

\begin{proof}
The fact that $\overline{\gp}$ is a presentation for $\m{\nonfdt}{\biauto}$ follows from the comments made before the statement of the lemma. By considering the (left-to-right) length-plus-lexicographic ordering
on $\{a,b,c\}^*$ induced by $a>b>c$ one sees that the rewriting system
$\overline{\gp}$ is noetherian.

The set of irreducible words under this
rewriting system is  the set $A^* \cup c^+b^*a^*$. Indeed, if a irreducible
word contains a symbol $c$, it cannot be to the left of a symbol from $A$,
otherwise we could apply a relation $K_a$ or $K_b$. Moreover, if the word
also contains a symbol $b$, then all symbols $a$ must be to the right of the
rightmost symbol $b$, since otherwise we could use a relation
of the form $\overline{C}_u$.

Finally, to prove that $\overline{\gp}$ is confluent it suffices to consider all possible overlaps between left-hand sides of the rewriting rules $K_x$ ($x \in A$) and
$\overline{C}_u$ ($u \in A^*$), showing that all critical peaks arising from these overlaps resolve (see \cite[Section~2]{gray_propertiesnotinherited}). There are three different ways in which these rewrite rules can overlap, giving rise to three types of critical peaks, all of which can be resolved; see \fullref{Figure}{Figure_resolutions}. This proves that $\overline{\gp}$ is confluent and thus completes the proof of the lemma.
\end{proof}

\begin{lemma}\label{lem_bio}
The monoid $\m{\nonfdt}{\biauto}$ is \biauto.
\end{lemma}

\begin{proof}
Let $L = A^* \cup c^+b^*a^*$. We will prove that $(A\cup\{c\},L)$ is a
  biautomatic structure for $\m{\nonfdt}{\biauto}$. By the previous lemma, $L$ is a regular
  language such that every element of $\m{\nonfdt}{\biauto}$ has a unique representative in $L$.
  Hence
  \[
    L_\emptyword = \gset{(w,w)}{w\in L};
  \]
  thus $L_\emptyword$ is a rational relation.

  Now let $u \in L$. Regardless of whether $u \in A^*$ or $u \in c^+b^*a^*$, the word $ua$ also lies in $L$. Hence
  \[
    L_a = \gset{(u,ua)}{u \in L} = \gset{(u,u)}{u \in L}(\emptyword,a)
  \]
  is a rational relation by \fullref{Proposition}{prop:idrational}.

  If $u \in A^*$, then $ub$ also lies in $L$. On the other hand, if $u = c^ib^ja^k$, then
  $ub = c^ib^ja^kb \reduces c^ib^{j+k+1}$ via a sequence of applications of rules $\overline{C}_u$. Hence
  \begin{align*}
    L_b &= \gset{(u,ub)}{u \in A^*} \cup \gset{(c^ib^ja^k,c^ib^{j+k+1})}{i \geq 1, j,k \geq 0} \\
    &= \gset{(u,u)}{u \in A^*}(\emptyword,b) \cup (c,c)^+(b,b)^*(a,b)^*(\emptyword,a)
  \end{align*}
  is a rational relation by \fullref{Proposition}{prop:idrational}.

  If $u=a^k$, then $uc = a^kc \reduces ca^k$ using a sequence of applications of rules $K_a$. If $u \in A^*$ and $u$
  contains at least one symbol $b$, then $u = u'ba^k$ for some $u' \in A^*$ and $k \in \nset\cup\set{0}$ and so $uc = u'ba^kc \reduces cu'ba^k \reduces cb^{|u'|+1}a^k$ by a sequences of applications of rules of $K_a$ and $K_b$ and a sequence of applications of rules $\overline{C}_u$. Hence
  \begin{align*}
    L_c ={}& \gset{(a^k,ca^k)}{k \geq 0}\\
           & \cup \gset{(uba^k,cb^{|u|+1}a^k)}{u \in A^*, k \geq 0} \\
           &\cup \gset{(c^ib^ja^k,c^{i+1}b^ja^k)}{i \geq 1, j,k \geq 0} \\
        ={}& (\emptyword,c)(a,a)^*\\
           &\cup (\emptyword,c)\set{(a,b),(b,b)}^*(b,b)(a,a)^* \\
           &\cup (\emptyword,c)(c,c)^+(b,b)^*(a,a)^*
  \end{align*}
  is a rational relation.

  Similar reasoning shows that
\begin{align*}
  {}_aL ={} & \gset{(u,au)}{u \in A^*}                                                         \\
            & \cup \gset{(c^ia^k,c^ia^{k+1})}{i\geq 1, k\geq 0}                                \\
            & \cup \gset{(c^ib^ja^k,c^ib^{j+1}a^k)}{i,j \geq 1, k \geq 0}                      \\
  ={} & (\emptyword,a)\set{(a,a),(b,b)}^*                                                \\
            & \cup (c,c)^+(a,a)^*(\emptyword,a) \\
            & \cup (c,c)^+(b,b)^+(\emptyword,b)(a,a)^*; \displaybreak[0]\\
  {}_bL ={} & \gset{(u,bu)}{u \in A^*} \\
            & \cup \gset{(c^ib^ja^k,c^ib^{j+1}a^k)}{i \geq 1, j,k \geq 0} \\
  ={} & (\emptyword,b)\set{(a,a),(b,b)}^* \\
            & \cup (c,c)^+(\emptyword,b)(b,b)^+(aa,a)^* \displaybreak[0]\\
  {}_cL ={}   & \gset{(a^k,ca^k)}{k \geq 0}; \\
            & \cup \gset{(uba^k,cb^{|u|+1}a^k)}{u \in A^*, k \geq 0}   \\
            & \cup \gset{(c^ib^ja^k,c^{i+1}b^ja^k)}{i \geq 1, j,k \geq 0} \\
  ={} & (\emptyword,c)(a,a)^* \\
            & \cup (\emptyword,c)\set{(a,b),(b,b)}^*(b,b)(a,a)^* \\
            & \cup (\emptyword,c)(c,c)^+(b,b)^*(a,a)^*;
\end{align*}
thus ${}_aL$, ${}_bL$, and ${}_cL$ are all rational.

Since $(u,v) \in {}_xL \cup L_x$ for any $x \in A \cup \emptyword$ implies $\abs[\big]{|u|-|v| \leq 1}$,
\fullref{Proposition}{prop:rationalbounded} shows that their images under $\textvisiblespace^\$$ and ${}^\$\textvisiblespace$ are regular. Hence $(A,L)$ is
a biautomatic structure for $\m{\nonfdt}{\biauto}$.
\end{proof}

Let $\Gamma$ denote the
derivation graph of $\gp$, and $\overline{\Gamma}$ the derivation graph of $\overline{\gp}$. Let $\Gamma_Z$ denote the  connected components of $\Gamma$ with
vertex set the set of all words in $A\cup\{c\}$ with at least two occurrences of the letter $c$.  Likewise let $\overline{\Gamma}_Z$ be the connected component of
$\overline{\Gamma}$ with the same vertex set as $\Gamma_Z$.

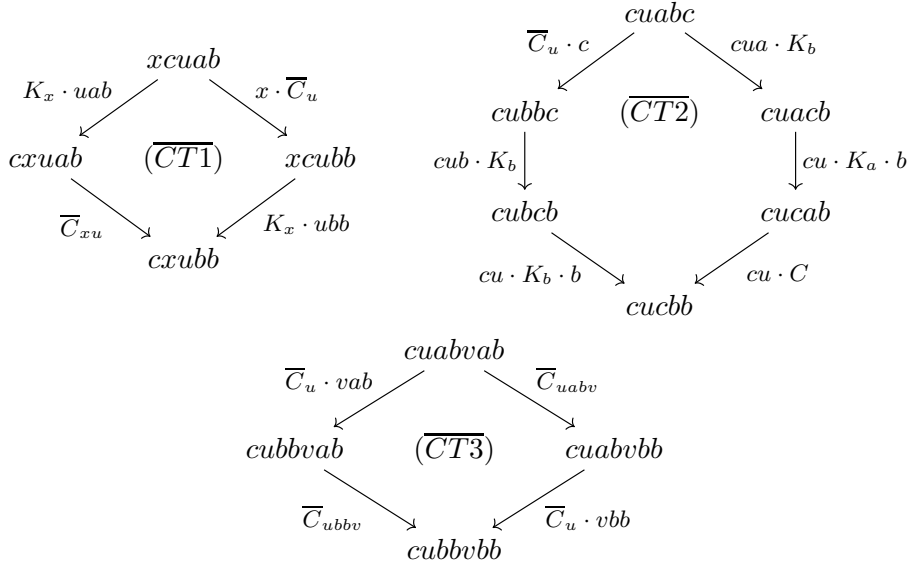
\begin{figure}[t]
\centering
\begin{tabular}{c@{\qquad}c}
\begin{tikzcd}[column sep=1.2em]
\null &	 xcuab \arrow[swap]{dl}{K_x \cdot uab} \arrow{dr}{x\cdot \overline{C}_{u}} & \\
cxuab \arrow[swap]{dr}{\overline{C}_{xu}} & \fdtpath{(\overline{CT1})} & xcubb \arrow{dl}{K_x\cdot ubb} \\
& cxubb &
\end{tikzcd}
&
\begin{tikzcd}[column sep=1.2em]
\null & cuabc \arrow[swap]{dl}{\overline{C}_{u} \cdot c}  \arrow{dr}{cua\cdot K_b} & \\
cubbc \arrow[swap]{d}{cub\cdot K_b} & \fdtpath{(\overline{CT2})} & cuacb \arrow{d}{cu\cdot K_a\cdot b}	\\
cubcb \arrow[swap]{dr}{cu\cdot K_b\cdot b  }	& & cucab \arrow{dl}{\ cu\cdot C} \\
& cucbb	&
\end{tikzcd}
\end{tabular}
\begin{tikzcd}[column sep=1.2em]
\null &	cuabvab \arrow[swap]{dl}{\overline{C}_{u}\cdot vab} \arrow{dr}{\overline{C}_{uabv}} & \\
cubbvab \arrow[swap]{dr}{\overline{C}_{ubbv}}	& \fdtpath{(\overline{CT3})} & cuabvbb \arrow{dl}{\overline{C}_{u}\cdot vbb} \\
& cubbvbb &
\end{tikzcd}
\caption{Resolutions of the critical peaks in the derivation graph $\overline{\Gamma}$ of the presentation $\overline{\gp}$. Here $x,y\in A$ and $u,v\in A^*$.}
\label{Figure_resolutions}
\end{figure}

There are three infinite families $\fdtpath{(\overline{CT1})}$,
$\fdtpath{(\overline{CT2})}$ and $\fdtpath{(\overline{CT3})}$ of closed paths
 in $\overline{\Gamma}$, as displayed in \fullref{Figure}{Figure_resolutions},
that correspond to resolutions of all critical peaks.  Each such
closed path we obtain we shall call a
\defterm{critical circuit}.
Let us denote by $\overline{\cal C}$ the critical circuits  of the form
$\fdtpath{(\overline{CT1})}$ and $\fdtpath{(\overline{CT3})}$
  and denote by $\overline{\cal
Z}$ the critical circuits of the form  $\fdtpath{(\overline{CT2})}$. Observe
that the critical circuits in  $\overline{\cal Z}$  are in
$\overline{\Gamma}_Z$ since the words labelling vertices in
$\fdtpath{(\overline{CT2})}$ all contain two occurrences of the letter $c$.
The set of critical circuits $\overline{\cal C}\cup \overline{\cal Z}$ forms an infinite homotopy base for $\overline{\Gamma}$ (see \cite[Lemma~3]{gray_propertiesnotinherited}).

We now want to use the infinite homotopy base $\overline{\cal C}\cup \overline{\cal Z}$ for $\overline{\Gamma}$ to obtain an infinite homotopy base for $\Gamma$. In order to do this we need to take the critical circuits $\fdtpath{(\overline{CT1})}$--$\fdtpath{(\overline{CT3})}$ and transform them into circuits in the derivation graph $\Gamma$ by replacing each occurrence of an edge $\overline{C}_u$ by a corresponding path in $\Gamma$.

As mentioned above when proving that $\gp$ and $\overline{\gp}$ are equivalent presentations,
the edges $\overline{C}_u$ can be realized in $\Gamma$ by paths $C_u$ which are defined inductively as follows: we first set  ${C}_{\epsilon}$ to be the rule $C$, and then for $u=xu'$, with $x\in A$ and $u'\in A^*$ we set $C_u$ to be the path
\begin{align}
\label{eqn_C}
C_u: \quad
cxu'a b
\xrightarrow{K_x^{-1} \cdot u' a b}
xcu'a b
\xrightarrow{x \cdot {C}_{u'}}
xcu'b b
\xrightarrow{K_x \cdot u' b b}
cxu'b b.
\end{align}
So, $C_u$ is the path in $\Gamma$ from $cuab$ to $cubb$ given by commuting $c$ through $u$ using the relations $K_x$, applying the relation $C$ to transform $ucab$ into $ucbb$, and then commuting $c$ back through $u$ again using the relations $K_x$, ending at the vertex $cubb$.

Now let us define a mapping $\varphi$ from the set of paths $P(\overline{\Gamma })$ in $\overline{\Gamma}$ to the set of paths  $P(\Gamma )$ in $\Gamma$.
Let $\varphi: P(\overline{\Gamma })\rightarrow P(\Gamma )$ be the map given by $(\alpha \cdot\overline{C}_u\cdot\beta) \varphi=\alpha \cdot C_u\cdot\beta$, for all $\alpha,\beta\in (A\cup\{c\})^*$ and $u\in A^*$, and defined to be the identity on every other edge of $\overline{\Gamma}$. Let ${\cal C} = (\overline{\cal C})\varphi$ and ${\cal Z}=(\overline{\cal Z})\varphi$.
Since $\overline{\cal C}\cup \overline{\cal Z}$ forms a homotopy base for  $\overline{\Gamma}$ it follows that ${\cal C}\cup {\cal Z}$ is an infinite homotopy base for $\Gamma$ (see \cite[Corollary~3.7]{squier_finiteness}).

Observe that the infinite homotopy base ${\cal C}\cup {\cal Z}$ for $\Gamma$ is nothing more than the set of circuits in $\Gamma$ obtained by taking the set of circuits $\fdtpath{(\overline{CT1})}$--$\fdtpath{(\overline{CT3})}$ and replacing each occurrence of the edge $\overline{C}_u$ by the path $C_u$ defined in \eqref{eqn_C}. Let us denote this corresponding set of circuits ${\cal C}\cup {\cal Z}$ in $\Gamma$ by $\fdtpath{({CT1})}$--$\fdtpath{({CT3})}$.

The monoid $\m{\nonfdt}{\biauto}$ is presented by the finite presentation $\gp$ and  $\Gamma$ is
the derivation graph of $\gp$ which has an infinite homotopy base  ${\cal
  C}\cup{\cal Z}$.

\begin{lemma}[{\cite[Discussion before Definition 3.3]{squier_finiteness}}]
  \label{lem:finitesubset}
  If $\m{\nonfdt}{\biauto}$ were \fdt\ then there would be a finite subset ${\cal C}_0\cup{\cal Z}_0$ of
  ${\cal C}\cup{\cal Z}$ which would be a finite homotopy base for $\Gamma$.
\end{lemma}

\begin{proof}[Sketch]
  Since $\mathcal{D}$ is a finite homotopy base, each path in $\mathcal{D}$ is homotopic to an empty path using finitely
  many paths from ${\cal C}\cup{\cal Z}$, and thus the finite subset of ${\cal C}\cup{\cal Z}$ consisting of all paths
  arising this way is a homotopy base.
\end{proof}

Our aim now
is to show that this leads to a contradiction, and thus conclude that $\m{\nonfdt}{\biauto}$ is not
\fdt. In order to do this we shall now define a mapping from the set of paths
$P(\Gamma)$ in $\Gamma$ into the integral monoid ring $\mathbb{Z}\monly$. Define
$\Phi: P(\Gamma) \rightarrow \ZM$ to be the unique map which extends the
mapping:
\begin{itemize}
\item $(\alpha \cdot K_a \cdot \beta)\Phi = \overline{\alpha}$, \quad $(\alpha \cdot K_b \cdot \beta)\Phi = \overline{\alpha}$, \quad and
\item $(\alpha \cdot C \cdot \beta)\Phi = \overline{\alpha }$,
\end{itemize}
where $\alpha,\beta \in (A\cup\{c\})^*$
and
$\overline{\alpha} \in \m{\nonfdt}{\biauto}$ denotes the element of $\m{\nonfdt}{\biauto}$ represented by the word $\alpha$,
to paths
in such a way that
\[
(\bbp \circ \bbq)\Phi = (\bbp)\Phi + (\bbq)\Phi
\;\; \mbox{and} \;\;
(\bbp^{-1})\Phi = - (\bbp)\Phi.
\]
The following basic properties of $\Phi$ are then easily verified for all paths $\bbp, \bbq \in P(\Gamma)$ and words $\alpha,\beta \in (A\cup\{c\})^*$:
\begin{enumerate}
\item $(\alpha \cdot \bbp \cdot \beta)\Phi = \overline{\alpha} \cdot(\bbp)\Phi  $
\item $(\bbp \circ \bbp^{-1})\Phi = 0$
\item $([\bbp,\bbq])\Phi = 0$ where
\[
[\bbp,\bbq] =
(\bbp \cdot \iota \bbq)
\circ
(\tau \bbp \cdot \bbq)
\circ
(\bbp^{-1} \cdot \tau \bbq)
\circ
(\iota \bbp \cdot \bbq^{-1}).
\]
\item If $\bbp \sim_0 \bbq$ then $(\bbp)\Phi = (\bbq)\Phi$.
\end{enumerate}
Here, property~4 follows from properties~1, 2 and~3. Note that property~4 implies that $\Phi$ induces a well-defined map on the homotopy classes of paths of $\Gamma$.

In what follows we shall often omit bars from the top of words in the images under $\Phi$ and simply write words from $(A\cup\{c\})^*$ with the obvious intended meaning. Recall that $\overline{\cal C}$ is the set of critical circuits  of the form $\fdtpath{(\overline{CT1})}$ and $\fdtpath{(\overline{CT3})}$, and that $\mathcal{C}=(\overline{\cal C})\varphi $ is the corresponding set of circuits in $\Gamma$.
Let $F=A^*$ denote the free monoid on the alphabet $A$.

\begin{lemma}
\label{lem_passtoZG}
If $\m{\nonfdt}{\biauto}$ is \fdt\ then the submodule $\lb (\mathcal{C})\Phi \rb_{\mathbb{Z}F}$, of the left $\mathbb{Z}F$-module $\mathbb{Z}F$,
generated by $(\mathcal{C})\Phi$
is a finitely generated left $\mathbb{Z}F$-module.
\end{lemma}
\begin{proof}
Assume that  $\m{\nonfdt}{\biauto}$ is \fdt\ and therefore $\gp$ is \fdt. Since $\mathcal{C}\cup\mathcal{Z}$ is a homotopy base for its derivation graph $\Gamma$, by \fullref{Lemma}{lem:finitesubset}
there are finite subsets $\mathcal{C}_0 \subseteq \mathcal{C}$ and  $\mathcal{Z}_0 \subseteq \mathcal{Z}$ such that
$\mathcal{C}_0\cup \mathcal{Z}_0$ is a finite homotopy base for $\Gamma$.
 Let $\bbp \in \mathcal{C}$ be
arbitrary. We claim that $(\bbp)\Phi \in \lb (\mathcal{C}_0)\Phi
\rb_{\mathbb{Z}F}$. Once established, this will prove the lemma, since
$(\mathcal{C}_0)\Phi$ is a finite subset of $\lb (\mathcal{C})\Phi
\rb_{\mathbb{Z}F}$.

By \cite[Lemma~2]{gray_propertiesnotinherited}, since $\bbp$ is a closed path in $\Gamma$ and
$\mathcal{C}_0\cup\mathcal{Z}_0$ is a homotopy base for $\Gamma$,
we can write
\begin{align}
\label{eqn_pathdecomp}
\bbp \sim_0
\bbp_1^{-1} \circ (\alpha_1 \cdot \bbq_1 \cdot \beta_1) \circ \bbp_1
\circ \cdots \circ
\bbp_n^{-1} \circ (\alpha_n \cdot  \bbq_n \cdot  \beta_n) \circ \bbp_n,
\end{align}
where each $\bbp_i \in P(\Gamma)$, $\alpha_i, \beta_i \in (A\cup\{c\})^*$ and $\bbq_i \in (\mathcal{C}_0\cup \mathcal{Z}_0)^{\pm 1}$.

Since the vertices of $\bbp$ have exactly one $c$, and all the relations
in the presentation $\gp$ involve the letter $c$, it follows that $\alpha_i,
\beta_i \in A^*$ and $\bbq_i\in \mathcal{C}_0$, for all $i\in\{1,\ldots,n\}$.
Applying $\Phi$  gives
$$ (\bbp)\Phi= \alpha_1(\bbq_1)\Phi + \cdots +\alpha_n(\bbq_n)\Phi\in \lb
(\mathcal{C}_0)\Phi \rb_{\mathbb{Z}F}$$ as claimed.
\end{proof}

To complete our proof, it remains to compute the subset $(\mathcal{C})\Phi$ of
$\mathbb{Z}F$ and then prove that the submodule $\lb (\mathcal{C})\Phi
\rb_{\mathbb{Z}F}$ of $\mathbb{Z}F$ is not finitely generated as a left
$\mathbb{Z}F$-module.
Recall that $\overline{\cal C}$ is the set of critical circuits  of the form $\fdtpath{(\overline{CT1})}$ and $\fdtpath{(\overline{CT3})}$, and that $\mathcal{C}=(\overline{\cal C})\varphi $.
So  $\mathcal{C}$ is the set of closed paths $\fdtpath{(CT1)}$ and
$\fdtpath{(CT3)}$ obtained by applying the mapping $\varphi$ to the closed paths
$\fdtpath{(\overline{CT1})}$ and $\fdtpath{(\overline{CT3})}$, that is, obtained
by taking each occurrence of $\overline{C}_u$ and replacing it by the path
$C_u$.

From  equation \eqref{eqn_C} we have, for any word $u \in A^*$, if
$u=xu'$, with $x\in A$,  the equality $(C_u)\Phi=-1_\monly+\overline{x}\cdot
(C_{u'})\Phi +1_\monly=\overline{x}\cdot (C_{u'})\Phi $, and thus
we deduce that $(C_{u})\Phi = u$. Using this fact, the result of computing
 a critical circuit in $\mathcal{C}$ under the map $\Phi$   is
given by:
\begin{itemize}
 \item $(\bbp)\Phi=0_{\mathbb{Z}\monly}$, for a critical circuit $\bbp$ of the
family $\fdtpath{(CT1)}$;
\item $(\bbp)\Phi=u(ab-bb)v$, for a critical circuit $\bbp$ of the
family $\fdtpath{(CT3)}$.
\end{itemize}

%
%

The next lemma completes the proof of \fullref{Theorem}{thm_exampleNotFDT}.

\begin{lemma}
The submodule $\lb (\mathcal{C})\Phi \rb_{\mathbb{Z}F}$ of $\mathbb{Z}F$, where
\[
(\mathcal{C})\Phi=\{ u(ab-bb)v \mid u,v\in A^*\}\cup\{0_{\mathbb{Z}\monly}\},
\]
is not finitely generated as a left $\mathbb{Z}F$-module, and therefore $\m{\nonfdt}{\biauto}$ is not \fdt.
\end{lemma}
\begin{proof}
Suppose, with the aim of obtaining a contradiction, that $\lb (\mathcal{C})\Phi \rb_{\mathbb{Z}F}$ is finitely generated as a left $\mathbb{Z}F$-module.
Then there exists a finite subset $X$ of $(\mathcal{C})\Phi$ such that  $\lb X \rb_{\mathbb{Z}F}\supseteq (\mathcal{C})\Phi $.
Let $n\in \mathbb{N}$ be the maximal length of a word $v$ where $u(ab-bb)v\in X$.
We shall show that $(ab-bb)a^{n+1}$ belongs to $(\mathcal{C})\Phi$ but not to $\lb X \rb_{\mathbb{Z}F}$.
Suppose that in $\mathbb{Z}F$ we have \[\displaystyle (ab-bb)a^{n+1}=\sum_{i=1}^{k}\alpha_i(ab-bb)v_i,\] where $\alpha_i, v_i\in A^*$ and $|v_i|\leq n$, for $i=1,\ldots, k$.
Then for some $j\in\{1,\ldots, k\}$ we have either
$bba^{n+1} = \alpha_jabv_j$ or
$bba^{n+1} = \alpha_jbbv_j$ in the free monoid $F=\{a,b\}^*$, which
clearly contradicts the fact that $|v_i| \leq n$.
We conclude that $(\mathcal{C})\Phi$
is not finitely generated as a left $\mathbb{Z}F$-module, and it
then follows from \fullref{Lemma}{lem_passtoZG} that $\m{\nonfdt}{\biauto}$ is not \fdt.
\end{proof}

\section{Free products of homogeneous monoids}
\label{sec_freeproducts}

In \fullref{Section}{sec:fundexs}, we gave four examples of
homogeneous monoids that possess certain combinations of the
properties \fcrs, \fdt, \biauto, and \auto. In this section, we use
free products to construct examples with the remaining consistent
combinations of these properties. Note that if monoids $M_1$ and $M_2$
have homogeneous presentations $\pres{A_1}{\rel{R}_1}$ and
$\pres{A_2}{\rel{R}_2}$, then their free product $M_1\ast M_2$ is defined by
the presentation $\pres{A_1 \cup A_2}{\rel{R}_1 \cup \rel{R}_2}$ and
is thus also homogeneous.

First we consider the interaction of the free product with \biauto\ and \auto. It is known that the free product of two
monoids is \auto\ if and only if each of the monoids is \auto\ (\cite[Theorem 6.2]{campbell_autsg}, \cite[Theorem
1.2]{duncan_change}). While it would be possible to extend this result to biautomaticity for general monoids, this
generalization does not appear in the literature. This paper required the biautomaticity result only for homogeneous
monoids, and the proofs are simpler in this case:

\begin{proposition}
\label{prop:homogenousautofreeproduct}
Let $M_1$ and $M_2$ be homogeneous monoids. Then $M_1 \ast M_2$ is
\auto\ (respectively, \biauto) if and only if $M_1$ and $M_2$ are
\auto\ (respectively, \biauto).
\end{proposition}

\begin{proof}
Let $\pres{A_1}{\rel{R}_1}$ and $\pres{A_2}{\rel{R}_2}$ be homogeneous
presentations for $M_1$ and $M_2$, respectively.

Suppose $M_1 \ast M_2$ is \auto; the proof for \biauto\ is similar. By
\fullref{Proposition}{prop:autochangegen}, $M_1 \ast M_2$ admits an
automatic structure $(A_1 \cup A_2,L)$. Since $M_1$ is homogeneous,
there is no non-empty word over $A_1$ representing the identity of
$M_1$; hence every word in $(A_1\cup A_2)^*$ representing an element
of $M_2$ must lie in $A_2^*$. Thus $(A_2,L \cap A_2^*)$ is an
automatic structure for $M_2$; similarly, $(A_1,L\cap A_1^*)$ is an
automatic structure for $M_1$.

On the other hand, suppose $M_1$ and $M_2$ are \auto; again, the proof for \biauto\ is similar. Then there are automatic
structures $(A_1,L_1)$ and $(A_2,L_2)$ for $M_1$ and $M_2$ respectively. By \cite[Corollary~5.5]{campbell_autsg}, assume
without loss of generality that every element of $M_1$ and $M_2$ has a unique representative in $L_1$ and $L_2$,
respectively. Let
\begin{align}
  L ={}& \set{\emptyword} \label{eq:homogenousautofreeproduct1}\\
  & \cup \parens[\Big]{\parens[\big]{L_1 \cap (A_1^+ \times A_1^+)}\parens[\big]{L_2 \cap (A_2^+ \times A_2^+)}}^+ \label{eq:homogenousautofreeproduct2}\\
  & \cup \parens[\Big]{\parens[\big]{L_1 \cap (A_1^+ \times A_1^+)}\parens[\big]{L_2 \cap (A_2^+ \times A_2^+)}}^*\parens[\big]{L_1 \cap (A_1^+ \times A_1^+)} \label{eq:homogenousautofreeproduct3}\\
  & \cup \parens[\Big]{\parens[\big]{L_2 \cap (A_2^+ \times A_2^+)}\parens[\big]{L_1 \cap (A_1^+ \times A_1^+)}}^+ \label{eq:homogenousautofreeproduct4}\\
  & \cup \parens[\Big]{\parens[\big]{L_2 \cap (A_2^+ \times A_2^+)}\parens[\big]{L_1 \cap (A_1^+ \times A_1^+)}}^*\parens[\big]{L_2 \cap (A_2^+ \times A_2^+)};\label{eq:homogenousautofreeproduct5}
\end{align}
note that $L$ is a disjoint union of the languages \eqref{eq:homogenousautofreeproduct1}--\eqref{eq:homogenousautofreeproduct5}.
(Note that the language $L$ is not equal to $(L_1 L_2)^*$: every word in $L$ is a product of words strictly alternating
between non-empty words from $L_1$ and words from $L_2$.)

Since every element of $M_1 \ast M_2$ is represented by a unique element of $L$,
\[
L_\emptyword = \gset{(u,u)}{u \in L}
\]
is a rational relation.

Now let $a \in A_1$, and let $v_a$ be the unique word in $L_1$ that is equal to $a$ (note that $v_a$ will either be $a$
itself or another generator that is equal to $a$). Suppose that $u \in L$ lies in \eqref{eq:homogenousautofreeproduct2}
or \eqref{eq:homogenousautofreeproduct5}. Then $uv_a$ is the unique word in $L$ that is equal to $ua$. On the other
hand, if $u \in L$ lies in \eqref{eq:homogenousautofreeproduct3} or \eqref{eq:homogenousautofreeproduct4}, then it is of
the form $u's$, where $s \in L_2$ and $s$ is the maximal suffix of $u$ lying in $A_1^*$. In this case, the unique word
in $L$ equal to $ua$ is $u't$, where $t$ is the unique word in $L_1$ equal to $sa$: that is, where $(s,t) \in (L_1)_a$. Hence
\begin{align*}
  L_a ={}& \set{(\emptyword,v_a)} \\
         & \cup \parens[\big]{(L_1)_\emptyword \cap (A_1^+\times A_1^+)}\parens[\big]{(L_2)_\emptyword \cap (A_2^+\times A_2^+)}(\emptyword,v_a) \\
         & \cup \parens[\Big]{\parens[\big]{(L_1)_\emptyword \cap (A_1^+\times A_1^+)}\parens[\big]{(L_2)_\emptyword \cap (A_2^+\times A_2^+)}}^*\parens[\big]{(L_1)_a \cap(A_1^+ \times A_1^+)} \\
         & \cup \parens[\big]{(L_2)_\emptyword \cap (A_2^+\times A_2^+)}\parens[\Big]{\parens[\big]{(L_1)_\emptyword \cap (A_1^+\times A_1^+)}\parens[\big]{(L_2)_\emptyword \cap (A_2^+\times A_2^+)}}^*\parens[\big]{(L_1)_a \cap(A_1^+ \times A_1^+)} \\
         & \cup \parens[\Big]{\parens[\big]{(L_2)_\emptyword \cap (A_2^+\times A_2^+)}\parens[\big]{(L_1)_\emptyword \cap (A_1^+\times A_1^+)}}^*\parens[\big]{(L_2)_\emptyword \cap (A_2^+\times A_2^+)}\parens[\big]{(L_1)_a \cap(A_1^+ \times A_1^+)}.
\end{align*}
Thus $L_a$ is a rational relation.

Since $M_1\ast M_2$ is homogeneous, if $(u,v) \in L_x$ for $x \in A$, then $|v| = |ua| = |u| + 1$. Furthermore, if
$(u,v) \in L_\emptyword$, then $|u| = |v|$. Hence $L_x^\$$ and ${}^\$L_x$ are regular for all
$x \in A \cup \set{\emptyword}$ by \fullref{Proposition}{prop:rationalbounded}. Thus $M_1 \ast M_2$ is automatic.

(This reasoning is essentially \cite[Proof of Theorem~6.2]{campbell_autsg}, but simplified because we consider only
homogeneous semigroups.)
\end{proof}

Now we consider the interaction of free product with \fcrs\ within the
class of homogeneous monoids.

\begin{theorem}[{\cite[Theorem~D]{pride_rewriting}}]
\label{thm:fcrsfreefactorifnounits}
Let $M_1$ and $M_2$ be monoids. Suppose that $M_2$ has no non-trivial
left- or right-invertible elements. If $M_1 \ast M_2$ is \fcrs,
then $M_1$ and $M_2$ are \fcrs.
\end{theorem}

\begin{proposition}
\label{corol:fcrsfreeprodhomog}
Let $M_1$ and $M_2$ be homogeneous monoids. Then $M_1 \ast M_2$ is \fcrs\ if and only if
$M_1$ and $M_2$ are \fcrs.
\end{proposition}

\begin{proof}
Suppose $M_1 \ast M_2$ is \fcrs. Since $M_1$ and $M_2$ are
homogeneous, neither contains any non-trivial left- or
right-invertible elements. So $M_1$ and $M_2$ are both \fcrs\ by
\fullref{Theorem}{thm:fcrsfreefactorifnounits}.

The converse part follows by \cite[Proposition
4.3]{otto_modular_tr}.
\end{proof}

Finally, we recall the following result on the
interaction of the free
product and \fdt:

\begin{theorem}[\cite{otto_modular,otto_modular_tr}]
\label{thm:fdtfreeprod}
Let $M_1$ and $M_2$ be monoids. Then $M_1 \ast M_2$ is \fdt\ if and
only if $M_1$ and $M_2$ are \fdt.
\end{theorem}

Now, the examples in \fullref{Section}{sec:fundexs} suffice to construct the remaining examples in
\fullref{Figure}{fig:examplerelationships} by taking their free products. Regarding
\fullref{Figure}{fig:examplerelationships} as a semilattice, examples in \fullref{Section}{sec:fundexs} correspond to
the elements of this semilattice that have no non-trivial decomposition.
\fullref{Proposition}{corol:fcrsfreeprodhomog}, \fullref{Theorem}{thm:fdtfreeprod}, and
\fullref{Proposition}{prop:homogenousautofreeproduct} together show that by taking a free product, one obtains a new
monoid whose properties are given by taking the logical conjunction (that is, the `and' operation) of the corresponding
properties. That is, given two monoids $\m{\aprop}{\bprop}$ and $\m{\cprop}{\dprop}$, their free product
$\m{\aprop}{\bprop} \ast \m{\cprop}{\dprop}$ will have the weaker of the two properties \aprop\ and \bprop\ (which lie in $\set{\text{\fcrs},\text{\fdt},\text{\nonfdt}}$) and the weaker of the two properties \cprop\ and \dprop\ (which lie in $\set{\text{\biauto},\text{\auto},\text{\nonauto}}$). Thus, to summarize:
\begin{itemize}
\item $\m{\fcrs}{\auto}\ast\m{\fdt}{\biauto}$ is non-\fcrs, \fdt, non-\biauto, \auto;
\item $\m{\fdt}{\biauto}\ast\m{\fcrs}{\nonauto}$ is non-\fcrs, \fdt, non-\biauto, non-\auto;
\item $\m{\fcrs}{\auto}\ast\m{\nonfdt}{\biauto}$ is non-\fcrs, non-\fdt, non-\biauto, \auto;
\item $\m{\fcrs}{\nonauto}\ast\m{\nonfdt}{\biauto}$ is non-\fcrs, non-\fdt, non-\biauto, non-\auto.
\end{itemize}

\begin{theorem}
\label{thm:homogeneous}
For each consistent combination of the properties \fcrs, \fdt, \biauto, \auto, and their negations, there exists a
homogeneous monoid with exactly that combination of properties.
\end{theorem}

\section{From homogeneous to $n$-ary multihomogeneous monoids}
\label{sec:extending}

Thus far we have proved that for every consistent combination of the properties \fcrs, \fdt, \biauto, \auto, and their
negations, there exists a homogeneous monoid with exactly those properties. In the remainder of the paper, we show that
such monoids exist in the more restricted class of $n$-ary homogenous monoids, and show that monoids with some
consistent combinations of these properties exist in the even more restricted class of $n$-ary multihomogeneous monoids. The
current section describes the overall strategy and results; the following two sections then develop the necessary
concepts and techniques. First, we introduce and investigate the theory of abstractly Rees-commensuable semigroups in
\fullref{Section}{sec:abstractrees}, ultimately proving \fullref{Corollary}{cor_multi}, which implies that from each
homogeneous monoid listed in \fullref{Table}{tb:resume_properties}, we can obtain an $n$-ary homogeneous monoid with the
same combination of properties. Thus we have the following analogy of \fullref{Theorem}{thm:homogeneous} for $n$-ary homogeneous
monoids:

\begin{theorem}
\label{thm:from_homog_to_n-ary_homog}
For each consistent combination of the properties \fcrs, \fdt, \biauto,
\auto, and their negations, there exists an $n$-ary homogeneous monoid
with exactly that combination of properties.
\end{theorem}

Since we have examples of $n$-ary homogeneous monoids with all consistent combinations of the properties \fcrs, \fdt,
\biauto, \auto, and their negations, the next step is to extend these examples to $n$-ary multihomogeneous
monoids. With that aim, in \fullref{Section}{sec:embedding} we define and investigate an embedding of a ($n$-ary)
homogeneous monoid into a $2$-generated ($n$-ary) multihomogeneous monoid. As stated below in
\fullref{Corollary}{cor:from_homog_to_multihomog}, passing to and from the ($n$-ary) multihomogeneous monoid preserves \fdt,
\auto, and \biauto. Furthermore, passing \emph{to} the ($n$-ary) multihomogeneous monoid preserves \fcrs. It is unknown whether
\fcrs\ is preserved passing back to the original ($n$-ary) homogeneous monoid, or, equivalently, whether non-\fcrs\ is
preserved on passing to the ($n$-ary) multihomogeneous monoid. Applying this embedding technique to the list of $n$-ary
homogeneous monoids we discussed above, we get the following results:

\begin{theorem}
  \label{thm:from_homog_to_multihomog1}
  For each consistent combination of the properties \fcrs, \biauto, \auto, and their negations, there exists an $n$-ary
  multihomogeneous monoid with exactly that combination of properties.
\end{theorem}

\begin{theorem}
  \label{thm:from_homog_to_multihomog2}
  For each consistent combination of the properties \fdt, \biauto, \auto, and their negations, there exists an $n$-ary
  multihomogeneous monoid with exactly that combination of properties.
\end{theorem}

To obtain an analogue of \fullref{Theorem}{thm:from_homog_to_n-ary_homog} for $n$-ary multihomogeneous monoids it would
be sufficient to find a multihomogeneous monoid that is non-\fcrs, \fdt, and \biauto.  Indeed, in that case, combining
\fullref{Theorems}{thm:from_homog_to_multihomog1} and \ref{thm:from_homog_to_multihomog2} with the results of
\fullref{Section}{sec_freeproducts} and noting that a free product of multihomogeneous monoids is again
multihomogeneous, we would get, for any consistent combination of \biauto, \auto, and their negations, an example of an
\fdt, non-\fcrs\ multihomogeneous monoid with exactly that combination of properties. Therefore, from
\fullref{Corollary}{cor_multi} we would get examples of $n$-ary multihomogeneous monoids with exactly the same discussed
properties. Joining these examples with the examples from \fullref{Theorems}{thm:from_homog_to_multihomog1} and
\ref{thm:from_homog_to_multihomog2} we would get the intended result. Thus we have the following question:

\begin{question}\label{question:multi_fdt_non-fcrs}
  Does there exist a multihomogeneous monoid that is \fdt, but non-\fcrs?
\end{question}

\section{Abstractly Rees-commensurable semigroups}
\label{sec:abstractrees}

We introduce a new definition which is inspired by the notions of \defterm{abstractly com\-men\-su\-ra\-ble groups}
\cite[\S\S~iv.27ff.]{delaharpe_geometric} and \defterm{Rees index} for semigroups \cite{Jura1978}. A subsemigroup $T$ of
a given semigroup $S$ has finite Rees index if $S\setminus T$ is finite. In that case the semigroup $S$ is said to be a
\defterm{small extension} of $T$, and $T$ a \defterm{large subsemigroup} of $S$. The main interest is that large
subsemigroups and small extensions of a given semigroup share many important properties of that semigroup (see
\cite{cm_finreessurvey} for a survey).

\begin{definition}
Two semigroups $S_1$ and $S_2$ are said to be
\defterm{abstractly Rees-com\-men\-su\-ra\-ble} if there are finite Rees index
subsemigroups $T_i\subseteq S_i$ (for $i=1,2$) with $T_1 \cong T_2$.
\end{definition}

Is is easy to verify that abstract Rees-com\-men\-su\-ra\-bility is an
equivalence relation on semigroups.

This notion can be naturally extended to ideals.
\begin{definition}
Two semigroups $S_1$ and $S_2$ are said to be \defterm{abstract
Rees-ideal-com\-men\-su\-ra\-ble} if there are finite Rees index ideals
$U_i\subseteq S_i$ (for $i=1,2$) with $U_1 \cong U_2$.
\end{definition}

The idea behind these notions is that abstract
Rees-ideal-com\-men\-su\-ra\-ble semigroups share many important properties,
such as \fcrs, \fdt, \biauto, and \auto.

\begin{proposition}
\fcrs, \fdt, \biauto, and \auto\ are preserved under abstract Rees-ideal-com\-men\-su\-ra\-bility:
\end{proposition}

\begin{proof}
\begin{itemize}
\item It is known that \fcrs\ is inherited by small extensions
  \cite[Theorem~1]{wang_fcrs} and by large subsemigroups
  \cite[Theorem~1.1]{wong_fcrs}. Although the result for small extensions was
stated
  in the context of monoids, it can be naturally extended to
  semigroups. These two results imply that \fcrs\ is preserved under
  abstract Rees(-ideal)-com\-men\-su\-ra\-bility.

\item It is known that \fdt\ is inherited by small extensions
  \cite[Theorem~2]{wang_fcrs} of monoids, and by large semigroup
  \emph{ideals} \cite[Theorem~1]{malheiro_fdtlargeideals}. We recall that the notion of
  finite derivation type was first introduced for monoids, but it was
  naturally extended to the semigroup case \cite[Section~2]{malheiro_fdtreesmatrix}.

The result on small extensions can be easily adapted for the semigroup
case. Indeed, let $T$ be a semigroup and consider the monoid $T^1$
obtained from $T$ by adding an identity. If $T$ is \fdt\ the derivation
graph of $T^1$ can be obtained from the derivation graph of $T$ by
adding an extra connected component with a single vertex corresponding
to the empty word.  Thus $T^1$ is \fdt. Now, if $S$ is a small
extension of the semigroup $T$, it turns out that the monoid $S^1$ is
a small extension of the monoid $T^1$. Therefore, by \cite[Theorem
2]{wang_fcrs} the monoid $S^1$ is \fdt. But $S$ is a large
ideal of $S^1$, and by \cite[Theorem 1]{malheiro_fdtlargeideals} we conclude
that $S$ is \fdt.

The two results on small extensions and large ideals of semigroups
show that \fdt\ is preserved under abstract
Rees-ideal-com\-men\-su\-ra\-bility.

\item By \cite[Theorem 1.1]{HTR02} and its natural analogue for \biauto,
  both \auto\ and \biauto\ are inherited by small extensions and by large
  subsemigroups, and therefore \auto\ and \biauto\ are preserved under
  abstract Rees(-ideal)-com\-men\-su\-ra\-bility.
\end{itemize}
\end{proof}


The preceding result is important because, in the case of
(multi)homogeneous and $n$-ary (multi)homogeneous monoids, the
following result holds:

\begin{proposition}
Every finitely presented (multi)homogeneous monoid is
Rees-ideal-commen\-su\-ra\-ble to an $n$-ary (multi)homogeneous
monoid, where $n$ can be chosen arbitrarily as long as it is greater
than or equal to the length of the longest relation in $\rel{R}$.
\end{proposition}

\begin{proof}
  Let $\pres{A}{\rel{R}}$ be a finite homogeneous presentation of a monoid $M$.  Let $n$ be chosen arbitrarily, provided
  it is greater than or equal to the maximum length of a relation in $\rel{R}$. We are going to construct an $n$-ary
  (multi)homogeneous monoid $M'$ and show it is abstract Rees-ideal-commensurable to $M$ by finding isomorphic ideals
  $I$ and $I'$ of $M$ and $M'$ respectively.

Let $I=\{[u]_M\in M : |u|\geq n\}$. Note that
$I$ is an ideal of finite Rees index in $M$, since its complement in $M$ is the
finite set  $\bigcup_{0\leq i< n}\{[a_1]_M\cdots [a_i]_M:a_1,\ldots,
a_i\in A\}$.

Let $\rel{R}'=\{(u\ell v,urv): (\ell ,r)\in \rel{R}, u,v\in A^*, |u\ell v|=n\}$.
Consider the $n$-ary
(multi)homogeneous presentation $\pres{A}{\rel{R}'}$, and let $M'$ be the
monoid defined by it.

The set $I'=\{[u]_{M'}\in M': |u|\geq n\}$ is an ideal of $M'$. Moreover, $I'$
has finite Rees index in $M'$.

Since $\rel{R}'$ is contained in the Thue congruence generated by
$\rel{R}$, we can define a map $\varphi: I' \rightarrow I$, with
$\left([u]_{M'}\right)\varphi =[u]_{M}$. Given $[u]_M$ in $I$, thus with
$u\in A^*$ and $|u|\geq n$, we have $[u]_{M'}\in I'$ by definition of $I'$.
Therefore $\varphi$ is surjective. It is also injective since,
for any
$[u]_{M'}$, $[v]_{M'}\in I'$ such that
$\left([u]_{M'}\right)\varphi
=\left([v]_{M'}\right)\varphi$, we get
$[u]_{M}=[v]_{M}$, with $|u|=|v|\geq n$, and therefore
$[u]_{M'}=[v]_{M'}$. It is routine to check that $\varphi$
is a homomorphism. Thus $I$ and $I'$ are isomorphic finite Rees index
ideals of $M$ and $M'$. So $M$ and $M'$ are abstract
Rees-ideal-com\-men\-su\-ra\-ble.
\end{proof}

\begin{corollary}\label{cor_multi}
Let $P$ be a set of properties preserved under abstract
Rees-ideal-com\-men\-su\-ra\-bility. Then there exists a (multi)homogeneous
monoid satisfying every property in $P$ if and only if there exists
an $n$-ary (multi)ho\-mo\-ge\-neous monoid satisfying every property
in $P$.
\end{corollary}

Note that the set of properties $P$ can contain `negative' properties like `not
finitely generated'. As an immediate consequence, we obtain
\fullref{Theorem}{thm:from_homog_to_n-ary_homog}

\section{From $n$-ary homogeneous to $n$-ary multihomogeneous monoids}
\label{sec:embedding}

We now develop the embedding technique that allows us to construct multihomogeneous examples. The technique embeds a
homogeneous monoid into a multihomogenous monoid that shares several of the properties under study. If the original
monoid is $n$-ary homogeneous, then the monoid it is embedded into is $n$-ary multihomogenous, so we use the
formulations `($n$-ary) homogeneous' and `($n$-ary) multihomogeneous' to indicate that statements apply to both the
general and $n$-ary cases.

Let us start by fixing some notation that will be maintained throughout this section. Let $M$ be a finitely generated
($n$-ary) homogeneous monoid, and let $\pres{A}{\rel{R}}$ be a ($n$-ary) homogeneous presentation defining $M$.  Recall
that $A = \{a_1,\ldots,a_m\}$ is minimal, so that any other generating set of $M$ contains $A$ (see discussion at the
end of \fullref{Subsection}{subsection_words}).

We define a homomorphism
\[
\phi : A^* \to \{x,y\}^*, \qquad a_i \mapsto x^{2}y^ixy^{m+1-i}.
\]
Denote by $N$ the finitely generated monoid presented by
$\pres{x,y}{\rel{R}\phi}$, where $\rel{R}\phi$ denotes the set
$\{(\ell\phi,r\phi ): (\ell,r)\in \rel{R}\}$.

\begin{proposition}
The monoid $N$, defined by the presentation $\pres{x,y}{\rel{R}\phi}$, is
($n(m+4)$-ary)
multihomogeneous.
\end{proposition}

\begin{proof}
The presentation $\pres{A}{\rel{R}}$ is homogeneous hence, for any $(u,v) \in
\rel{R}$, we have $|u| = |v|$. Since $a_i\phi$ contains $3$ symbols $x$
and $m+1$ symbols $y$ for all $i \in \{1,\ldots,m\}$, it follows that
$|u\phi|_x = 3|u| = 3|v| = |v\phi|_x$ and $|u\phi|_y = (m+1)|u| =
(m+1)|v| = |v\phi|_y$. Hence $\pres{x,y}{\rel{R}\phi}$ is
multihomogeneous.

If the presentation $\pres{A}{\rel{R}}$ is $n$-ary homogeneous, then also for
any relation $(u,v) \in
\rel{R}$, we have $|u| = |v|=n$. But in that case we get $|u\phi|=|u\phi|_x
+|u\phi|_y = 3|u|+ (m+1)|u|= 3n+(m+1)n= n(m+4)$, and similarly for $|v\phi|$.
\end{proof}

A subset $C$ of $\{x,y\}^*$ is a \defterm{code} if it is a set of free
generators for the submonoid of  $\{x,y\}^*$ generated by $C$
\cite[Subsection~7.2]{howie_fundamentals}. That is the case for the set
$A\phi=\{x^{2}y^ixy^{m+1-i}: i=1,\ldots,m\}$. Furthermore, $A\phi$ is what is
called a \defterm{prefix code}, since no word of $A\phi$ is a proper prefix of
another word in $A\phi$ \cite[Chapter~6]{lothaire_algebraic}.

By \cite[Proposition~6.1.3]{lothaire_algebraic}, since $\phi$ induces a
bijection from $A$ to the code $A\phi$, we conclude that $\phi$ is injective,
and for that reason $\phi$ is called a \defterm{coding morphism for} $A\phi$.

\begin{proposition}
\label{prop:homogenousembedding}
The monoid $M$ presented by $\pres{A}{\rel{R}}$ embeds into the monoid $N$
presented by
$\pres{x,y}{\rel{R}\phi}$ via the map $[u]_M \mapsto [u\phi]_N$, and the words
over $\{x,y\}$ representing elements of (the image of) $M$ are
precisely the words in $(A^*)\phi$.
\end{proposition}

\begin{proof}
 Consider the natural projection $\rho_N$ of $\{x,y\}^*$ onto $N$, whose kernel
is generated by $\rel{R}\phi$. Since $\phi$ is injective, the kernel of the
composition $\phi\circ\rho_N:A^*\rightarrow N$ is the congruence generated by
$\rel{R}$.

 The natural projection of $\rho_M$ of $A^*$ onto $M$ has the same kernel has
the map $\phi\circ\rho_N$. By \cite[Theorem 1.5.2]{howie_fundamentals} there
is a monomorphism $\overline{\phi}:M\rightarrow N$ such that
$(M)\overline{\phi}=(A^*)\phi\circ\rho_N$ and  $([u]_M)\overline{\phi}
=(u)\phi\circ\rho_N= [u\phi]_N$.
\end{proof}

We shall now turn our attention to investigating the relationship between
\fdt\ holding in $M$ and \fdt\ holding in $N$.
For this we shall prove some results which relate the Squier graphs of $M$
and $N$.

Let $\Delta$ denote the subgraph of $\Gamma(\pres{x,y}{R\phi})$
induced on the set of vertices $(A\phi)^*$.

Let us extend the mapping $\phi:A^*\rightarrow\{x,y\}^*$ to a mapping from
the
derivation graph $\Gamma(\pres{A}{\rel{R}})$ to the derivation graph
$\Gamma(\pres{x,y}{\rel{R}\phi})$ mapping an edge
$\bbe=(w_1,(\ell,r),\epsilon, w_2)$ to an edge
$\bbe\phi=(w_1\phi,(\ell\phi,r\phi),\epsilon,
w_2\phi)$.
 The mapping $\phi$ extends to a mapping between paths by putting
\[
(\bbe_1 \bbe_2 \ldots \bbe_k)\phi =
(\bbe_1 \phi)(\bbe_2 \phi) \ldots (\bbe_k \phi),
\]
for any edges $\bbe_1, \bbe_2, \ldots,\bbe_k$ of $\Gamma(\pres{A}{\rel{R}})$.

\begin{lemma} \label{lemma:multihomogeneous_bijection_derivation_graph}
The mapping $\phi: \Gamma(\pres{A}{\rel{R}}) \rightarrow
\Gamma(\pres{x,y}{\rel{R}\phi})$ has the following properties.
\begin{enumerate}
\item $\phi$ maps $A^*$ bijectively to $(A\phi)^*$;
\item For every edge $\bbe$ in $\Gamma(\pres{A}{\rel{R}})$ we have
\[
\iota (\bbe \phi) = (\iota \bbe) \phi, \quad \mbox{and} \quad
\tau (\bbe \phi) = (\tau \bbe) \phi;
\]
\item For every  edge $\bbe$ in $\Gamma(\pres{A}{\rel{R}})$ and every pair of
words $u,v \in A^*$ we have
\[
(u \cdot \bbe \cdot v)\phi = u \phi \cdot \bbe \phi \cdot v \phi;
\]
\item $\phi$ maps the edge-set of $\Gamma(\pres{A}{\rel{R}})$ bijectively to
the
edge-set of $\Delta$.
\end{enumerate}
\end{lemma}
\begin{proof}
\noindent (i) It follows from the fact that $\phi$ is an injective homomorphism
from $A^*$ into $\{x,y\}^*$, and thus $\phi$ maps bijectively $A^*$ to its
image $(A^*)\phi=(A\phi)^*$.

\noindent (ii) \& (iii) The identities follow from the definition of the
extended map $\phi $ on edges and from the fact that $\phi$ is a homomorphism.

\noindent (iv) The injectivity of the edge-set of $\Gamma(\pres{A}{\rel{R}})$
 to the edge-set of $\Delta$, follows from the injectivity of $\phi$.
Now, consider an edge $\bbf=(z_1 , (\ell\phi,r\phi), \epsilon,z_2)$  of
$\Delta$, and suppose without lost of generality, that $\epsilon =+1$.

The word  $\iota\bbf=z_1(\ell\phi)z_2$ decomposes as a concatenation of words
in $A\phi$; that is,
 words of the form $x^2y^jxy^{m+1-j}$ for various $j \in
 \{1,\ldots,m\}$. Since subwords $x^2$ only occur at the start of such
 words in $A\phi$, the $x^2$ at the start of $\ell\phi$ lies at the start
 of a word from $A\phi$ in the decomposition of $\iota\bbf$. Since all
 words in $A\phi$ have the same length, $\ell\phi$ must also finish at the
 end of some word from $A\phi$ in the decomposition of $\iota\bbf$. Hence
 $z_1$ and $z_2$ are (possibly empty) concatenations of words in $A\phi$,
 and so $z_1 = w_1\phi$ and $z_2 = w_2\phi$ for some $w_1,w_2 \in A^*$. Thus
$\bbf =\bbe\phi$, for the edge $\bbe =(w_1,(\ell,r),+1,w_2)$ of
$\Gamma(\pres{A}{\rel{R}})$.
\end{proof}

Let $\psi: (A\phi)^* \rightarrow A^*$ be the right inverse of the injective
mapping $\phi$. Note that $(w_1w_2)\psi=(w_1\psi)(w_2\psi)$, for any elements
$w_1,w_2\in (A\phi)^*$.
By part (iv) of the above lemma we can extend $\psi$ to a mapping of edges by
simply setting $(\bbe \phi) \psi = \bbe$ and
using the fact that every edge of $\Delta$ has the form $\bbe\phi$ for a unique
edge $\bbe$ from $\Gamma(\pres{A}{R})$.

\begin{lemma}
\label{lemma:multihomogeneous_bijection_psi_derivation_graph}
The mapping $\psi: \Delta \rightarrow \Gamma(\pres{A}{R})$ has the
following properties.
\begin{enumerate}
\item $\psi$ maps $(A\phi)^*$ bijectively to $A^*$;
\item For every edge $\bbe$ in $\Delta$  we have
\[
\iota (\bbe \psi) = (\iota \bbe) \psi, \quad \mbox{and} \quad
\tau (\bbe \psi) = (\tau \bbe) \psi;
\]
\item For every edge $\bbe$ in $\Delta$ and every pair of words $u,v \in
\{x,y\}^*$ we have that $u \cdot \bbe \cdot v$ is an edge in $\Delta$ if and
only if $u,v \in (A\phi)^*$ in which case
\[
(u \cdot \bbe \cdot v)\psi = u \psi \cdot \bbe \psi \cdot v \psi.
\]
\item $\psi$ maps the edge-set of $\Delta$ bijectively to the edge-set of
$\Gamma(\pres{A}{R})$.
\end{enumerate}
\end{lemma}
\begin{proof}
\noindent (i) Follows immediatly from the definition of $\psi$.

\noindent (ii) From
\fullref{Lemma}{lemma:multihomogeneous_bijection_derivation_graph} each edge of
$\Delta$ has the form $\bbe\phi=(w_1\phi, (\ell\phi,r\phi),\epsilon, w_2\phi)$,
with $w_1,w_2\in A^*$, $(\ell,r)\in \rel{R}$ and $\epsilon=\pm 1$. The image of
the edge under $\psi$  is $\bbe$.  Since $\phi$ is an homomorphism and from the
definition of $\psi$ we get $((w_1\phi)(\ell\phi)(w_2\phi))\psi= w_1\ell w_2$
 and $((w_1\phi)(r\phi)(w_2\phi))\psi= w_1r w_2$ as required.

\noindent (iii)  Suppose that  $\bbe$ is an edge in $\Delta$ and that for some
$u,v \in \{x,y\}^*$ we have $u \cdot \bbe \cdot v$ an edge in $\Delta$. So the
vertex $u(\iota \bbe)v$ of $\bbe$ is a word in $(A\phi)^*$. Arguing as in
\fullref{Lemma}{lemma:multihomogeneous_bijection_derivation_graph}~(iv) we can
conclude that $u,v\in (A\phi)^*$.

The converse part of the equivalence follows trivially.

If $\bbe$ in $\Delta$ and $u,v\in (A\phi)^*$, there exists $\bbe'$ in
$\Gamma(\pres{A}{R})$ and $u',v'\in A^*$, such that
$
(u \cdot \bbe \cdot v)= u' \phi \cdot \bbe \phi \cdot v' \phi$. From the
definition of $\psi$  we get $
(u \cdot \bbe \cdot v)\psi = u \psi \cdot \bbe \psi \cdot v \psi$.

\noindent (iv) It follows from the definition of $\psi$ and
\fullref{Lemma}{lemma:multihomogeneous_bijection_derivation_graph}.
\end{proof}

From the above lemmas, we can regard, up to the given encoding, the graph
$\Gamma(\pres{A}{\rel{R}})$ as a subgraph of $\Gamma(\pres{x,y}{\rel{R}\phi})$,
identifying  $\Gamma(\pres{A}{\rel{R}})$ with
$\Delta$.

\begin{lemma}\label{lemma:preserve_homotopy_multihomogeneous}
Let $\mathcal{C}$ be a set of closed paths in $\Gamma(\pres{A}{R})$. Then
$\mathcal{C}\phi$ is a set of closed paths in $\Delta$.
Moreover, for every closed path $\bbp$ in $\Gamma(\pres{A}{R})$, if
$\bbp \sim_\mathcal{C} 1_{\iota\bbp}$ in $\Gamma(\pres{A}{R})$
then $\bbp\phi \sim_{\mathcal{C}\phi} 1_{(\iota \bbp) \phi}$ in
$\Gamma(\pres{x,y}{R\phi})$.
\end{lemma}
\begin{proof}
The result follows by \cite[Theorem~3.6]{squier_finiteness}.
\end{proof}

At this point,  our strategy starts to emerge. The identification  between
$\Delta$ with the derivation
graph  $\Gamma(\pres{A}{\rel{R}})$, will allow us to get a finite homotopy base
of $\Gamma(\pres{A}{\rel{R}})$, from a finite homotopy bases of
$\Gamma(\pres{x,y}{\rel{R}\phi})$.

\begin{lemma}\label{lemma:preserve_homotopy_reverse_multihomogeneous}
Let $\mathcal{D}$ be a set of closed paths in $\Delta$. Then $\mathcal{D}\psi$
is a set of closed paths in $\Gamma(\pres{A}{R})$.
Moreover, for every closed path $\bbp$ in $\Delta$, if
$\bbp \sim_\mathcal{D} 1_{\iota\bbp}$ in $\Gamma(\pres{x,y}{R\phi})$
then $\bbp\psi \sim_{\mathcal{D}\psi} 1_{(\iota \bbp) \psi}$ in
$\Gamma(\pres{A}{R})$.
\end{lemma}
\begin{proof}
The result follows from \cite[Lemma 1]{Lafont1995}, where the author uses
strict monoidal categories to prove an analogue of
\cite[Theorem~3.6]{squier_finiteness}.

Indeed, in that context, both $\Gamma(\pres{A}{R})$ and $\Delta$ are
strict monoidal categories (notice that $\Delta$ is not a derivation graph,
and so we can not refer to \cite[Theorem~3.6]{squier_finiteness} to prove the
lemma).  The mapping $\psi$ is a functor that preserves the multiplicative
structure, by
\fullref{Lemma}{lemma:multihomogeneous_bijection_psi_derivation_graph}.
The homotopy relation $\sim_{\mathcal{D}\psi}$ is what is called a
$2$-congruence on $\Delta$, satisfying $\bbp\psi \sim_{\mathcal{D}\psi}
1_{(\iota \bbp) \psi}$, for any closed path $\bbp$ in $\mathcal{D}$. From
\cite[Lemma 1]{Lafont1995}, it follows that $\bbp\psi \sim_{\mathcal{D}\psi}
1_{(\iota \bbp) \psi}$, for any closed path $\bbp$ of $\Delta$ such that $\bbp
\sim_\mathcal{D} 1_{\iota\bbp}$.
\end{proof}

Next, we describe the remaining connected
componnets of $\Gamma(\pres{x,y}{\rel{R}\phi})$.

\begin{lemma} \label{lemma:properties_derivation_graph_xy}
The derivation graph $\Gamma(\pres{x,y}{\rel{R}\phi})$ has the
following properties:
 \begin{enumerate}
 \item If a vertex $z$ has maximal factors (with respect to length)  $u$ and
$v$ in $(A\phi)^*$, then either $z=z_0uz_2=z_0vz_2$, or $z=z_0uz_1vz_2$, or
$z=z_0vz_1uz_2$, with $z_0,z_1,z_2\notin (A\phi)^+$ and $z_1$ non-empty;
  \item Any vertex has a
unique decomposition of the form \[w_0
u_1 w_1 \cdots w_{k-1} u_k w_k,\] where each $u_i$ is a maximal factor
in $(A\phi)^*$, and each $w_i$ has no factor in $(A\phi)^*$;
\item Any edge  has the form
$z_1\cdot(\bbe\phi)\cdot z_2$, for some edge $\bbe$ in
$\Gamma(\pres{A}{\rel{R}})$, where $\iota\bbe\phi$ and $\tau\bbe\phi$ are
maximal factors in $(A\phi)^*$ of the corresponding initial and terminal
vertices of $\bbe\phi$, and $z_1,z_2$ are not in $(A\phi)^+$.
 \end{enumerate}
\end{lemma}
\begin{proof}
\noindent (i) Let $u$ and $v$ be maximal factors in $(A\phi)^*$ of $z$. Suppose
$u$ and $v$ overlap in the word $z$. Without lost of generality, suppose that
$z=u_0uu_1=v_0vv_1$ with $|u_0|\leq |v_0|< |u_0u|$. Since $v\in (A\phi)^*$
is a sequence of words of the form $x^2y^jxy^{m+1-j}$ for various $j \in
 \{1,\ldots,m\}$, and subwords $x^2$ only occur at the start of such
 words in $A\phi$, the $x^2$ at the start of $v$ lies at the start
 of a word from $A\phi$ in the decomposition of $u$. By the maximality of $v$
and since $z=u_0uu_1=v_0vv_1$ we get $u_0=v_0$.
  Since all
 words in $A\phi$ have the same length, $u$ and $v$ must also finish at the
 end of some word from $A\phi$. By the maximality of $u$ and $v$ we get $u=v$.

\noindent (ii) The result follows from (i).

\noindent (iii) From the definition of $\Gamma(\pres{x,y}{\rel{R}\phi})$, an
edge has for initial and terminal vertices words $t$ and $s$ in $\{x,y\}^*$
where one is obtained from the other
by a
single application of a defining relation in $\rel{R}\phi$, that is,
$s = w_0 (\ell\phi) w_1$ and $t = w_0(r\phi)w_1$,
with $(\ell,r)\in\rel{R}$. Notice that $\ell\phi, r\phi\in(A\phi)^*$. So,
we can consider maximal factors $p$ and $q$ in $(A\phi)^*$ having as factors
$\ell\phi$ and $r\phi$, respectively, in such a way that
\begin{align*}
s &= z_1 (p\phi) z_2, \\
t &= z_1(q\phi)z_2.
\end{align*}
The words $p$ and $q$ are the initial and terminal vertices of an edge $\bbe$
in $\Gamma(\pres{A}{\rel{R}})$. The result follows as in the statement.
\end{proof}

\begin{lemma}\label{lemma:consecutive_edges_multihomogeneous}
A length-two path in $\Gamma(\pres{x,y}{\rel{R}\phi})$ has one of
the forms:
\begin{enumerate}
 \item $(z_0\cdot\bbe_1\phi\cdot z_2)\circ(z_0\cdot\bbe_2\phi\cdot z_2)$, and
$\bbe_1\circ\bbe_2$ is a path in $\Gamma(\pres{A}{\rel{R}})$; or
 \item $(z_0\cdot\bbe_1\phi\cdot z_1(\iota \bbe_2\phi)z_2)\circ
(z_0(\tau\bbe_1\phi)z_1\cdot \bbe_2\phi\cdot z_2)$; or
 \item $(z_0(\iota\bbe_2\phi)z_1\cdot \bbe_1\phi\cdot
z_2)\circ(z_0\cdot\bbe_2\phi\cdot z_1(\tau \bbe_1\phi)z_2)$,
\end{enumerate}
for $z_0,z_1,z_2\notin (A\phi)^+$ and $z_1$ non-empty, and $\bbe_1,\bbe_2$
edges in $\Gamma(\pres{A}{\rel{R}})$.
\end{lemma}
\begin{proof}
 By
\fullref{Lemma}{lemma:properties_derivation_graph_xy}~(iii), the path of
consecutive edges has the form $(w_1\cdot(\bbe_1\phi)\cdot
w_2)\circ(z_1\cdot(\bbe_2\phi)\cdot z_2)$. In the conditions of that lemma the
words $\tau\bbe_1\phi$ and $\iota\bbe_2\phi$ are maximal factors in $(A\phi)^*$
of the same word $w_1(\tau\bbe_1\phi)w_2=z_1(\iota\bbe_2\phi)z_2$. From
\fullref{Lemma}{lemma:properties_derivation_graph_xy}~(i) we have three
possible cases. Each of the cases corresponds to one the possible cases (i),
(ii) and (iii) of the statement.
\end{proof}

\begin{lemma}\label{lemma:direct_part_fdt_multihomog}
 Let $\bbp$ be a non-empty (closed) path in
$\Gamma(\pres{x,y}{\rel{R}\phi})$. Then
\begin{equation}\label{path_ripas}
 \bbp \sim_0 \bbp_1\circ \bbp_2 \circ \cdots \circ\bbp_k,
 \end{equation}
for some $k\in \mathbb{N}$, and where  each $\bbp_i$ has the form
\[
w_0(\tau \bbq_1\phi) w_1 \cdots (\tau \bbq_{i-1}\phi) w_{i-1} \cdot
(\bbq_i\phi)\cdot w_i (\iota\bbq_{i+1}\phi) \cdots w_{k-1}(\iota\bbq_k\phi)w_k,
\]
for some (closed) path $\bbq_i$ in
$\Gamma(\pres{A}{\rel{R}})$, with $w_0,\ldots,w_k\notin (A\phi)^*$ and
$w_1,\ldots,w_{k-1}$ non-empty.
\end{lemma}
\begin{proof}
The key observation to prove the lemma is to note that in
\fullref{Lemma}{lemma:consecutive_edges_multihomogeneous} the paths on items
(ii) and (iii) are $\sim_0$-homotopic. Indeed, this is a consequence of
disjoint derivations.

By \fullref{Lemma}{lemma:properties_derivation_graph_xy}~(ii) the initial
vertex of $\bbp$ can be factorized in the form \[w_0
u_1 w_1 \cdots w_{k-1} u_k w_k,\] where each $u_i$ is a maximal factor
in $(A\phi)^*$ and each $w_i$ has no factor in $(A\phi)^*$. Notice, also by the
same lemma,  that any edge in $\bbp$ has a factorization of the form
$z_1\cdot(\bbe\phi)\cdot z_2$, for some edge $\bbe$ in
$\Gamma(\pres{A}{\rel{R}})$, where $\iota\bbe\phi$ and $\tau\bbe\phi$ are
maximal factors in $(A\phi)^*$ of the corresponding initial and terminal
vertices of $\bbe\phi$.

Since $w_i$'s have no factor in $(A\phi)^*$, no relation from $\rel{R}\phi$ is
going to
be applied to them, and hence they are fixed in this path. Also, for each
$i\in\{1,\ldots,k\}$ we can identify $B_i$ as the set of edges in $\bbp$ where
the relation is applied to a word between $w_{i-1}$ and $w_i$.

From \fullref{Lemma}{lemma:consecutive_edges_multihomogeneous} any two
consecutive edges $\bbe_i\circ\bbe_j$, with $\bbe_i\in B_i$ and $\bbe_j\in
B_j$, and $i> j$, we get $\bbe_i\circ\bbe_j\sim_0\bbe_j\circ\bbe_i$ by
disjoint derivations.

Consequently, taking $i=1$ we can group together all edges in $B_1$ by finding
a $\sim_0$-homotopic path in which all those edges are at the beginning of the
path. Proceeding in this way with the remaining edges we get the intended
result.
\end{proof}

The following result can be viewed as a generalization of Lemma~2.3 in
\cite{kobayashi_onerel}.

\begin{proposition} \label{prop:multi_fdt}
The monoid $M$ is \fdt\ if and only if the monoid
$N$ is \fdt.
\end{proposition}

\begin{proof}
Suppose that $M$ is \fdt\ and let $\mathcal{C}$ be  a finite homotopy base
(of closed paths) for $\Gamma(\pres{A}{\rel{R}})$. We shall see that the finite
set
$\mathcal{C}\phi=\{\bbp\phi : \bbp\in \mathcal{C}\}$ of
closed paths is a homotopy base for $\Gamma(\pres{x,y}{\rel{R}\phi})$ .

Let $\bbp$ be a non-empty closed path in
$\Gamma(\pres{x,y}{\rel{R}\phi})$. By
\fullref{Lemma}{lemma:direct_part_fdt_multihomog} there are
closed paths $\bbq_i$, for $i=1,\ldots,k$, in $\Gamma(\pres{A}{\rel{R}})$ such
that
\[
\bbp \sim_0 \bbp_1\circ \bbp_2 \circ \cdots \circ\bbp_k,
\]
and each $\bbp_i =w_i\cdot\bbq_i\phi\cdot w'_i$, for some $w_i,w_i'\in
\{x,y\}^*$.

Since $\mathcal{C}$ is a homotopy base, for each
$i=1,\ldots,k$, we have $\bbq_i\sim_{\mathcal{C}}1_{\iota(\bbq_i)}$. Thus
$\bbp_i\sim_{\mathcal{C}\phi} 1_{\iota(\bbp_i\phi)}$ by
\fullref{Lemma}{lemma:preserve_homotopy_multihomogeneous}, for all
$i=1,\ldots,k$, which in turn implies
that $\bbp\sim_{\mathcal{C}\phi} 1_{\iota(\bbp\phi)}$.  Consequently,
$\mathcal{C}\phi$ is
a finite homotopy base for $\Gamma(\pres{x,y}{\rel{R}\phi})$,  and so $N$ is
\fdt.

Conversely, suppose that $N$ is \fdt, and let $\mathcal{E}$ be a finite
homotopy base (of closed paths) for
$\Gamma(\pres{x,y}{\rel{R}\phi})$. Again by
\fullref{Lemma}{lemma:direct_part_fdt_multihomog},  for each
$\bbp$ in $\mathcal{E}$, there are closed paths $\bbq_1,\ldots ,\bbq_k$ in
$\Gamma(\pres{A}{\rel{R}})$, such that
\[ \bbp \sim_0 \bbp_1\circ\bbp_2 \circ \cdots \circ\bbp_k,
 \] and  $\bbp_i =w_i\cdot\bbq_i\phi\cdot w'_i$, for some $w_i,w_i'\in
\{x,y\}^*$. Denote by
$\mathcal{C}$ the finite set of all $\bbq_i$'s, for all $\bbp\in
\mathcal{E}$, and let $\mathcal{C}\phi$ denote the set $\{\bbp\phi : \bbp\in
\mathcal{C}\}$. From the definition of $\mathcal{C}\phi$, the paths $\bbp_i$
satisfy $\bbp_i\sim_{\mathcal{C}\phi}1_{\iota(\bbp_i)}$, and so each closed
path $\bbp$ in $\mathcal{E}$ satisfies
$\bbp\sim_{\mathcal{C}\phi}1_{\iota(\bbp)}$.  Thus,  the set
$\mathcal{C}\phi $ generates the homotopy relation
${\sim_{\mathcal{E}}}$, and therefore  also
$\mathcal{C}\phi$ is a finite homotopy base of
$\Gamma(\pres{x,y}{\rel{R}\phi})$.
Note that $\mathcal{C}\phi$ is a set of closed paths in $\Delta$ and that
$\mathcal{C}\phi\psi=\mathcal{C}$.


To conclude the proof, let $\bbp$ be a closed path in
$\Gamma(\pres{A}{\rel{R}})$. Since $\mathcal{C}\phi$ is a homotopy base of
$\Gamma(\pres{x,y}{\rel{R}\phi})$, we have
$\bbp\phi\sim_{\mathcal{C}\phi}1_{\iota(\bbp\phi)}$. By
\fullref{Lemma}{lemma:preserve_homotopy_reverse_multihomogeneous}, we get
$\bbp\phi\psi\sim_{\mathcal{C}\phi\psi}1_{\iota(\bbp\phi\psi)}$. Since
$\bbp\phi\psi=\bbp$ we get
$\bbp\sim_{\mathcal{C}}1_{\iota(\bbp)}$  as required.
Therefore,
$\mathcal{C}$ is a finite homotopy base of $\Gamma(\pres{A}{\rel{R}})$, and so
$M$ is \fdt.
\end{proof}

\begin{proposition}
\label{prop:homogenousembeddingauto}
The monoid $M$ is \auto\ (respectively,
\biauto) if and only if the monoid $N$ is
\auto\ (respectively, \biauto).
\end{proposition}
\begin{proof}
We will prove the result for
\biauto; the result for \auto\ follows by considering
multiplication only on one side.

Suppose that $N$ is \biauto. By
\fullref{Proposition}{prop:autochangegen}, there is a biautomatic
structure $(\{x,y\},K)$ for $N$. By
\fullref{Proposition}{prop:homogenousembedding}, words over $\{x,y\}$
representing elements of the image under $\phi$ of $M$ are precisely
those in $(A\phi)^*$. So $K \cap (A\phi)^*$ must map onto the image of
$M$.

The map $\phi$ is a rational relation, since
\[
\phi = \gset{(a_i,x^2y^ixy^{m+1−i})}{a_i \in A}^*.
\]
Thus its converse $\phi^{-1}$ is also a rational relation. Let $L = K\phi^{-1}$; note that $L \subseteq A^*$. Since
$\phi^{-1}$ is a rational relation and $K$ is regular, it follows that $L$ is also regular. Since $K \cap (A\phi)^*$
must map onto the image of $M$, the language $L$ maps onto $M$. For any $a \in A \cup \{\emptyword\}$,
\begin{align*}
(u,v) \in L_a &\iff u \in L \land v \in L \land ua =_M v \\
&\iff (\exists u',v' \in K)\parens[\big]{(u,u') \in \phi \land (v,v') \in \phi \land u'(a\phi) =_N v'} \\
&\iff (\exists u',v' \in K)\parens[\big]{(u,u') \in \phi \land (v,v') \in \phi \land (u',v') \in K_{a\phi}} \\
&\iff (u,v) \in \phi \circ K_{a\phi} \circ \phi^{-1};
\end{align*}
thus $L_a = \phi \circ K_{a\phi} \circ \phi^{-1}$. Since $(\{x,y\},K)$ is an automatic structure for $N$, each relation
$K_{a\phi}$ is rational. Since $\phi$ and $\phi^{-1}$ are rational relations, it follows that $L_a$ is a rational
relation. Since $M$ is homogeneous, $(u,v) \in L_a \implies \bigl||u|-|v|\bigr| \leq 1$, and so $L_a^\$$ and
${}^\$L_a$ are regular by \fullref{Proposition}{prop:rationalbounded}. Symmetrical reasoning shows that ${}_aL^\$$ and
${}_a^\$L$ are regular. Hence $(A,L)$ is a biautomatic structure for $M$.

Now suppose that $M$ is \biauto. By
\fullref{Proposition}{prop:autochangegen}, there is a biautomatic
structure $(A,L)$ for $M$. The aim is to construct a biautomatic structure for $N$.

Let $J = \{x,y\}^+ - \{x,y\}^*A^+\phi\{x,y\}^*$; so $J$ consists of all non-empty words over $\set{x,y}^*$ that do not
contain subwords equal to elements $M\phi$. Notice in particular that $\emptyword \notin J$ and that $J$ is closed under
taking non-empty subwords. Note further that $J$ is regular. For future use, let $\Delta = \gset{(z,z)}{z \in J}$; note
that $\Delta$ is a rational relation by \fullref{Proposition}{prop:idrational}.

Consider a word $w \in \set{x,y}^+$. By \fullref{Lemma}{lemma:properties_derivation_graph_xy}, $w$ can be uniquely
factored into an alternating product of words from $J$ and words from $A^+\phi$. That is,
\begin{equation}
  \label{eq:homogenousembeddingauto1}
  w = z_0u_1z_1\cdots z_{k-1}u_kz_k,
w\end{equation}
where each $u_i$ lies in $A^+\phi$ and each $z_i$ lies in $J$, except that $z_0$ and $z_k$ may also be
$\emptyword$. Note that each $u_i$ is equal to some word in $L\phi - \set{\emptyword}$. The idea is to build an automatic structure where
the language of representatives consists of alternating products of words from $J$ and words from $L\phi$. That is, the language will consist of words   \eqref{eq:homogenousembeddingauto1} where each $u_i$ lies in $L\phi - \set{\emptyword}$.

More formally, let
\[
K = (J \cup \set{\emptyword})\parens[\big]{(L\phi \cap \set{x,y}^+)J}^*(J \cup \set{\emptyword});
\]
note that $K$ is regular since $J$ is regular and $\phi$ is a homomorphism, and that every element of $N$ is equal to
some word in $K$ by the reasoning in the previous paragraph.

Now let $w \in K$ and factor $w$ as \eqref{eq:homogenousembeddingauto1} where each $u_i$ lies in
$L\phi - \set{\emptyword}$ and each $z_i$ lies in $J$, except that $z_0$ and $z_k$ may also be
$\emptyword$.

Consider right-multiplication of $w$ by a generator $x$. Since $z_kx$ cannot contain a subword from $A^+\phi$ (since
$z_k$ contains no such subword, and such subwords do not end with symbols $x$), the words in $L$ to which $wx$ is equal
are precisely words of the form $z_0u'_1z_1\cdots z_{k-1}u'_kz_kx$, where each $u'_i$ is any word in $L\phi$ that is
equal to $u_i$. (Recally that defining relations in $\rel{R}\phi$ only apply to words in $A^+\phi$; thus the subwords
$z_i$ are fixed.) Hence
\[
  K_x = (\Delta \cup \set{(\emptyword,\emptyword)})\parens[\big]{(\phi^{-1} \circ L_\emptyword \circ \phi)\Delta}^*\parens[\big]{L_\emptyword\phi(\emptyword,x) \cup (\emptyword,x)}
\]
is a rational relation.

Now consider right-multiplication of $w$ by a generator $y$. The $z_ky$ may end with a subword $x^2y^ixy^{m-i+1} \in A\phi$, and
thus it is necessary to distinguish three cases:
\begin{itemize}
\item $z_k = tx^2y^ixy^{m-i}$, where $t \neq \emptyword$. In this case, the factorization of $wy$ into an alternating
  product of subwords from $J$ and $A^+\phi$ is
  \[
    w = z_0u_1z_1\cdots z_{k-1}u_ktx^2y^ixy^{m-i};
  \]
  note that $t \in J$ since $J$ is closed under taking subwords. The words in $L$ to which $wy$ is equal are precisely
  words of the form $z_0u'_1z_1\cdots z_{k-1}u'_kz_kv$, where each $u'_i$ is any word in $L\phi$ that is equal to $u_i$,
  and $v$ is any word in $L\phi$ that is equal to $x^2y^ixy^{m-i} = a_i\phi$. Let $H_i = \gset{p \in L}{p = a_i}$; then
  $H_i$ is regular by \cite[Proposition 3.1]{campbell_autsg}. (In fact, since $M$ is homogeneous, $H$ can contain only
  generators from $A$ and so must be finite.) Let
  \[
    K^{(1)}_y = (\Delta \cup m
\set{(\emptyword,\emptyword)})\parens[\big]{(\phi^{-1} \circ L_\emptyword \circ
\phi)\Delta}^*\bigcup_{i=1}^m\parens[\big]{\set{x^2y^ixy^{m-i}} \times H_i\phi}.
  \]
  Since $H_i\phi$ is regular, $\set{x^2y^ixy^{m-i}} \times H_i\phi$ is a rational relation by
  \fullref{Proposition}{prop:intersectionwithregular} and so $K^{(1)}_y$ is a rational relation since it is a
  concatenation of rational relations.

  Then $K^{(1)}_y$ is a rational relation that describes right-multiplication by a generator $y$ in this case.

\item $z_k = x^2y^ixy^{m-i}$. In this case, the factorization of $wy$ into an alternating
  product of subwords from $J$ and $A^+\phi$ is
  \[
    w = z_0u_1z_1\cdots z_{k-1}u_kx^2y^ixy^{m-i};
  \]
  note that the last factor is $u_kx^2y^ixy^{m-i} \in A^+\phi$. The words in $L$ to which $wy$ is equal are precisely
  words of the form $z_0u'_1z_1\cdots z_{k-1}v$, where each $u'_i$ is any word in $L\phi$ that is equal to $u_i$, and
  $v$ is any word in $L\phi$ that is equal to $u_kx^2y^ixy^{m-i}$. That is, $v$ can be any word such that
  $(u_k,v) \in L_{a_i}\phi$. Let
  \begin{align*}
    K^{(2)}_y ={}& (\Delta \cup \set{(\emptyword,\emptyword)})\parens[\big]{(\phi^{-1} \circ L_\emptyword \circ \phi)\Delta}^*\\
    &\qquad\cdot \bigcup_{i=1}^n\parens[\big]{(\phi^{-1}\circ L_{a_i} \circ \phi)(x^2y^ixy^{m-i},\emptyword)}.
  \end{align*}
  Then $K^{(2)}_y$ is a rational relation that describes right-multiplication by a generator $y$ in this case.

\item $z_k$ does not end with a suffix of the form $x^2y^ixy^{m-i}$. The words in $L$ to which $wy$ is equal are precisely
  words of the form $z_0u'_1z_1\cdots z_{k-1}u'_kz_ky$, where each $u'_i$ is any word in $L\phi$ that is equal to $u_i$. Let
  \begin{align*}
    K^{(3)}_y ={}& (\Delta \cup \set{(\emptyword,\emptyword)})\parens[\big]{(\phi^{-1} \circ L_\emptyword \circ \phi)\Delta}^*(\phi^{-1} \circ L_\emptyword \circ \phi) \\
    &\quad \cdot \gset[\big]{(v,v)}{v \in J - \set{x,y}\bigcup_{i=1}^m
x^2y^ixy^{m-i}}(\emptyword,y).
  \end{align*}
  Then $K^{(2)}_y$ is a rational relation that describes right-multiplication by a generator $y$ in this case.
\end{itemize}
Thus $K_y = K^{(1)}_y \cup K^{(2)}_y \cup K^{(3)}_y$ is a rational relation.

Since $N$ is homogeneous, $(u,v) \in K_t \implies \bigl||u|-|v|\bigr| \leq 1$, and so $K_t^\$$ and ${}^\$K_t$ are
regular by \fullref{Proposition}{prop:rationalbounded}.

Similar reasoning shows that ${}_tK^\$$ and ${}_t^\$K$. Hence $(\set{x,y},K)$ is a biautomatic structure for $N$.
\end{proof}


\begin{proposition} \label{prop:homogenousembedding_fcrs}
If $\pres{B}{\rel{Q}}$ is a finite presentation for $M$, and so
$A\subseteq B$, then
\[
\mathcal{P}=\pres{x,y,B}{\rel{Q},\ (a\phi,a)\ (\forall a\in A)}
\]
is a finite presentation defining $N$. Moreover,
the presentation $\pres{B}{\rel{Q}}$ is complete if and only if the
presentation $\cal P$ is complete. Thus, if $M$ is
\fcrs\ then  $N$ is \fcrs.
\end{proposition}

\begin{proof}
Using Tietze transformations we obtain from the presentation
$\pres{x,y}{\rel{R}\phi}$ a new presentation for
$N$ as follows: for each $a\in A$, insert  a generator $a$ and
a relation $(a\phi,a)$,
thus obtaining a Tietze equivalent presentation
\[
\pres{x,y,A}{\rel{R}\phi,\ (a\phi,a)\ (\forall a\in A)}.
\]
Since $\phi$ is a homomorphism, identifying each symbol $a$ with the word
$a\phi$, performing these substitutions on the words from $\rel{R}\phi$, we
obtain another Tietze equivalent presentation defining the same monoid:
\[
\pres{x,y,A}{\rel{R},\ (a\phi,a)\ (\forall a\in A)}.
\]

As it is possible to obtain from the presentation $\pres{A}{\rel{R}}$
the presentation $\pres{B}{\rel{Q}}$ using finitely many Tietze
transformations, we can obtain from the presentation
$\pres{x,y,A}{\rel{R},\ (a\phi,a)\ (\forall a\in A)}$ the presentation
$\cal P$ by using the same Tietze transformations.

Suppose that $\pres{B}{\rel{Q}}$ is also complete. Observe that
$\rel{Q}$ relates words from the alphabet $B$, and that a relation
from the set $\rel{E}=\{(a\phi,a):a\in A\}$ has left-hand side in
$\{x,y\}^*$ and right-hand side in $A$. Thus, if
$w\imreduces_{\rel{Q}} w' \imreduces_{\rel{E}} w''$, we can find
$\overline{w}$ such that $w\imreduces_{\rel{E}} \overline{w}
\imreduces_{\rel{Q}} w''$. Indeed, since $w' \imreduces_{\rel{E}} w''$ the word
$w'$ has a factor from the alphabet $\{x,y\}^*$. Since $w\imreduces_{\rel{Q}}
w'$ the word $w'$ is obtained from $w$ by changing some factor in $B^*$. Thus,
the word $w$ also contains the left-hand side of the relation from $\rel{E}$
applied to $w'$. This means that $w$ has both left-hand sides of the relations
being applied, and those left-hand sides do not overlap. So,  we can
alternatively apply first the relation from $\rel{E}$, obtaining a word
$\overline{w}$, and then apply the relation from $\rel{Q}$, obtaining the word
$w''$.

Hence, $\imreduces_{\rel{E}}$
quasi-commutes over $\imreduces_{\rel{Q}}$, that is,
${\imreduces_{\rel{Q}}} \circ {\imreduces_{\rel{E}}} \subseteq
{\imreduces_{\rel{E}}} \circ {\reduces_{\rel{Q}\cup\rel{E}}}$. By
\cite[Theorem 1]{BacDers86}, the rewriting system $(\{x,y\} \cup A,
\rel{Q}\cup\rel{E})$ is terminating if, and only if, both $\rel{Q}$
and $\rel{E}$ are terminating. By assumption $\rel{Q}$ is terminating
and from where $\rel{E}$ is also terminating by length-reduction.

As before, since $\rel{Q}$ relates words from the alphabet $B$, and that
relations from  $\rel{E}$ have left-hand side in
$\{x,y\}^*$ and right-hand side in $A$,  we can also deduce that whenever
$w'\reducesrev_{\rel{Q}} w
\reduces_{\rel{E}} w''$, there exists $\overline{w}$ such that
$w'\reduces_{\rel{E}} \overline{w} \reducesrev_{\rel{Q}}
w''$. Therefore, the relations $\reduces_{\rel{E}}$ and
$\reducesrev_{\rel{Q}}$ commute, that is, ${\reducesrev_{\rel{Q}}}
\circ{\reduces_{\rel{E}}} \subseteq
{\reduces_{\rel{E}}}\circ{\reducesrev_{\rel{Q}}}$. By \cite[Lemma
  2.7.10]{baader_termrewriting}, if $\rel{E}$ and
$\rel{Q}$ are confluent and $\reduces_{\rel{E}}$ and
$\reducesrev_{\rel{Q}}$ commute, then
$\rel{Q}\cup \rel{E}$ is also confluent. Since by assumption $\rel{Q}$
is confluent it remains to show that $\rel{E}$ is confluent.

Observe that two left-hand sides $a_i\phi$ and $a_j\phi$ of rules in $\rel{E}$ can overlap if and only if
$a_i\phi=a_j\phi$. Since $\phi$ is injective we get $a_i=a_j$. Therefore, $\rel{E}$ has no critical pairs and thus is
also confluent.

Conversely, suppose that $\cal P$ is a complete presentation. Since
$\pres{B}{\rel{Q}}$ is contained in $\cal P$, we deduce that
$\pres{B}{\rel{Q}}$ is terminating. Confluence also
holds from the fact that in $\cal P$ all critical pairs are resolved,
in particular, those arising from relations in $\rel{Q}$. Hence, each
resolution associated to relations from $\rel{Q}$, can only involve
relations from $\rel{Q}$, since left-hand sides of rules in $\rel{E}$ belong to
$\{x,y\}^*$. Therefore, $\pres{B}{\rel{Q}}$ is complete.
\end{proof}

It is unknown if the property \fcrs\ is preserved from $N$ to $M$.  If this were true, we would have an example
satisfying the conditions of \fullref{Question}{question:multi_fdt_non-fcrs}.

Combining Propositions~\ref{prop:multi_fdt}, \ref{prop:homogenousembeddingauto}, and
\ref{prop:homogenousembedding_fcrs}, we conclude the following:

\begin{corollary}
  \label{cor:from_homog_to_multihomog}
  For each of the properties \fcrs, (non-)\fdt, (non-)\auto\ and (non-)\biauto, the ($n$-ary) homogeneous monoid $M$ has
  that property if and only if the ($n$-ary) multihomogeneous monoid $N$ has that property.
\end{corollary}

\bibliographystyle{elsarticle-num}
\bibliography{\jobname}

\end{document}